\documentclass[11pt,leqno]{amsart}

\usepackage{latexsym}
\usepackage{amssymb}
\usepackage{amsmath}
\usepackage{amsthm}
\usepackage{verbatim}
\usepackage{pdfsync}
\usepackage[margin=1.2in]{geometry}

\usepackage{color}
\usepackage{tikz}
\usepackage[all]{xy}

\numberwithin{equation}{section}  %%% this one seems to control what happens!!

\DeclareMathOperator*{\bbigwedge}{\text{\raisebox{0.25ex}{\scalebox{0.8}{$\bigwedge$}}}}

\begin{document}

%%% These commands yield the numbering conventions used in the Kyoto notes.
%\renewcommand{\thesection}{\Roman{section}}
%\renewcommand{\thesubsection}{\Roman{section}.\arabic{subsection}}

%new macros by Tonghai

\newcommand{\zxz}[4]{\begin{pmatrix} #1 & #2 \\ #3 & #4 \end{pmatrix}}
\newcommand{\abcd}{\zxz{a}{b}{c}{d}}
\newcommand{\kzxz}[4]{\left(\begin{smallmatrix} #1 & #2 \\ #3 &
#4\end{smallmatrix}\right) }
\newcommand{\kabcd}{\kzxz{a}{b}{c}{d}}

%%%%%%%%%%%%%%% This is the macros file for alll chapters of KYR
%%%%% last edited:  ssk  9/26/04

%%%%%%%%%%%%%%%  macros

%%%% mathbb

\newcommand{\A}{{\mathbb A}}
\newcommand{\C}{{\mathbb C}}
\newcommand{\F}{{\mathbb F}}
\newcommand{\G}{{\mathbb G}}
\newcommand{\R}{{\mathbb R}}
\newcommand{\Q}{{\mathbb Q}}
\newcommand{\X}{{\mathbb X}}
\newcommand{\Z}{{\mathbb Z}}
\newcommand{\HZ}{\widehat{\Z}}

%%%% mathrm

\newcommand{\rom}[1]{\text{\rm #1}}
\renewcommand{\roman}{\rm}

\newcommand{\Aut}{\text{\rm Aut}}
\newcommand{\CH}{\widehat{\text{\rm CH}}}
\newcommand{\cha}{{\text{\rm char}}}
\newcommand{\CHe}{\text{\rm CHeeg}}
\newcommand{\degh}{\widehat{\text{\rm deg}}}
\newcommand{\degH}{\widehat{\text{\rm deg}}}    %%% redundant def
\newcommand{\diag}{{\text{\rm diag}}}
\newcommand{\Diff}{\text{\rm Diff}}
\newcommand{\disc}{\text{\rm discr}}
\renewcommand{\div}{\text{\rm div}}
\newcommand{\divh}{\widehat{\text{\rm div}}}
\newcommand{\DS}{\text{\rm DS}}
\newcommand{\Ei}{\text{\rm Ei}}
\newcommand{\End}{\text{\rm End}}
\newcommand{\ev}{{\text{\rm ev}}}
\newcommand{\Gal}{\text{\rm Gal}}
\newcommand{\GL}{\text{\rm GL}}
\newcommand{\GSpin}{\text{\rm GSpin}}
\newcommand{\Hom}{\text{\rm Hom}}
\newcommand{\hor}{{\text{\rm horiz}}}
\newcommand{\id}{\text{\rm id}}
\newcommand{\im}{\text{\rm im}}
\renewcommand{\Im}{\text{\rm Im}}
\newcommand{\inv}{{\text{\rm inv}}}
\newcommand{\Jac}{\text{\rm Jac}}
\newcommand{\Leray}{{\mathrm L}}
\newcommand{\Lie}{\text{\rm Lie}}
\newcommand{\Mp}{\text{\rm Mp}}
\newcommand{\mult}{\text{\rm mult}}
\newcommand{\MW}{\text{\rm MW}}
\newcommand{\MWt}{\widetilde{\MW}}
\newcommand{\new}{\text{\rm new}}
\newcommand{\Nm}{\text{\rm Nm}}
\newcommand{\ord}{\text{\rm ord}}
\newcommand{\PGL}{\text{\rm PGL}}
\newcommand{\Pic}{\text{\rm Pic}}
\newcommand{\Pich}{\widehat{\text{\rm Pic}}}
\newcommand{\pr}{\text{\rm pr}}
\newcommand{\ra}{\text{\rm ra}}
\newcommand{\Rao}{\mathrm R}
\renewcommand{\Re}{\text{\rm Re}}
\newcommand{\sgn}{\text{\rm sgn}}
\newcommand{\sig}{\text{\rm sig}}
\newcommand{\SL}{\text{\rm SL}}
\newcommand{\SO}{\text{\rm SO}}
\newcommand{\Sp}{\text{\rm Sp}}
\newcommand{\Spec}{\text{\rm Spec}\, }
\newcommand{\Spf}{\text{\rm Spf}}
\newcommand{\supp}{\text{\rm supp}}
\newcommand{\Sym}{{\text{\rm Sym}}}
\newcommand{\tr}{\text{\rm tr}}
\newcommand{\type}{\text{\rm type}}
\newcommand{\Ver}{\text{\rm Vert}}
\newcommand{\vol}{\text{\rm vol}}
\newcommand{\Wald}{\text{\rm Wald}}

%%%% cals

\newcommand{\Cal}{\mathcal}     %%% this makes the old \Cal valid

\newcommand{\AHH}{\hat{\Cal A}}   % used??
\newcommand{\CHH}{\hat{\Cal C}}
\newcommand{\MM}{\Cal D}          % redefined!!!
\newcommand{\MMb}{\MM^\bullet}
\newcommand{\ssplit}{\text{\bf split}}
\newcommand{\whcc}{\widehat{\Cal C}}
\newcommand{\CO}{\mathcal O}
\newcommand{\COH}{\widehat{\CO}}
\newcommand{\M}{\Cal M}
\newcommand{\OB}{\Cal O_B}
\newcommand{\XX}{\mathcal X}
\newcommand{\bXX}{\bar\XX}
\newcommand{\wc}{\hat{\Cal C}}
\newcommand{\wch}{\wc^{\text{\rm hor}}}
\newcommand{\ZZ}{\Cal Z}
\newcommand{\ZH}{\widehat{\Cal Z}}   %%% redundant def's
\newcommand{\Zh}{\widehat{\Cal Z}}
\newcommand{\ZZh}{\ZZ^{\text{\rm hor}}}
\newcommand{\ZZv}{\ZZ^{\text{\rm ver}}}
\newcommand{\ZZhh}{\Zh^{\text{\rm hor}}}
\newcommand{\ZZhv}{\Zh^{\text{\rm ver}}}

%%%% math spacing

\newcommand{\nass}{\noalign{\smallskip}}
\newcommand{\snass}{\noalign{\vskip 2pt}}
\newcommand{\tent}[1]{ \vphantom{\vbox to #1pt{}} }   %%% !!!!

%%%% math fonts

\newcommand{\scr}{\scriptstyle}
\newcommand{\disp}{\displaystyle}

\font\cute=cmitt10 at 12pt
\font\smallcute=cmitt10 at 9pt
\newcommand{\kay}{{\text{\cute k}}}
\newcommand{\smallkay}{{\text{\smallcute k}}}

\renewcommand{\a}{\alpha}
\renewcommand{\b}{\beta}
\newcommand{\e}{\epsilon}
\renewcommand{\l}{\lambda}
\renewcommand{\L}{\Lambda}
\renewcommand{\o}{\omega}
\renewcommand{\O}{\Omega}
\renewcommand{\P}{\Phi}
\newcommand{\ph}{\varphi}
\newcommand{\phih}{\widehat{\phi}}
\newcommand{\wphi}{\widehat{\phi}}
\newcommand{\phit}{\widetilde{\phi}}
\newcommand{\s}{\sigma}
\newcommand{\vth}{\vartheta}

%%%% from Chapter VIII

%%\newcommand{\Gt}{\widetilde{G}}    %%%%  tilde's removed 6/20/04
%\newcommand{\Gt}{G}
%\newcommand{\Ph}{\Phi}
%\newcommand{\pht}{\widetilde{\phi}}
%%\newcommand{\Pht}{\widetilde{\Phi}}%%%%  tilde's removed 6/20/04
%\newcommand{\Pht}{\Phi}
%%\newcommand{\Pt}{\widetilde{P}}
%\newcommand{\Pt}{P}                 %%%%  tilde's removed 6/20/04
%%\newcommand{\Kt}{\widetilde{K}}    %%%%  tilde's removed 6/20/04
%\newcommand{\Kt}{K}
%%\newcommand{\It}{\widetilde{I}}    %%%%  tilde's removed 6/20/04
%\newcommand{\It}{I}
%\newcommand{\Jt}{\widetilde{J}}
%\newcommand{\lt}{\widetilde{\l}}
%\newcommand{\vp}{\varpi}
%

\newcommand{\Pt}{P}
\newcommand{\Ph}{\P}
\newcommand{\Pht}{\tilde \P}   %%%%%%%%    ****** conflicts *******
\newcommand{\Kt}{K}           %%%%  tilde's removed 6/20/04
\newcommand{\Mt}{M}
%%%%%%

%\newcommand{\Ph}{\Phi}      %%%%%%%%    ****** conflicts *******   temp %'ed
\newcommand{\pht}{\widetilde{\phi}}
\newcommand{\It}{I}
\newcommand{\Jt}{\widetilde{J}}
\newcommand{\lt}{\widetilde{\l}}
\newcommand{\vp}{\varpi}

\newcommand{\bom}{{\boldsymbol{\o}}}
\newcommand{\hbom}{\widehat{\bom}}
\newcommand{\ff}{{\bold f}}
\newcommand{\fsp}{\boldsymbol{f}_{\!\rm sp}}
\newcommand{\fev}{\boldsymbol{f}_{\!\rm ev}}
\newcommand{\fb}{\boldsymbol{f}}
\newcommand{\J}{\und{J}'}
\newcommand{\JJ}{\bold J'}
\newcommand{\V}{\bold V}
\newcommand{\xx}{\bold x}

\newcommand{\g}{{\mathfrak g}}
\renewcommand{\H}{\mathfrak H}

%%%%  math macros

\newcommand{\back}{\backslash}
\newcommand{\CT}[1]{\operatornamewithlimits{CT}_{#1}}
\renewcommand{\d}{\partial}
\newcommand{\db}{\bar\partial}
\newcommand{\dbar}{\bar{\partial}}
\newcommand{\gs}[2]{\langle \,#1,#2\,\rangle}
\newcommand{\Gt}{G}
\newcommand{\hfal}{h_{\text{\rm Fal}}}
\newcommand{\II}{\int^\bullet}
\newcommand{\isoarrow}{\ {\overset{\sim}{\longrightarrow}}\ }
\newcommand{\lisoarrow}{\ {\overset{\sim}{\longleftarrow}}\ }
\newcommand{\limdir}[1]{\underset{\underset{#1}{\rightarrow}}{\lim}}
\newcommand{\lan}{\operatorname{\langle}\hskip .5pt}
\newcommand{\ran}{\,\operatorname{\rangle}}
\newcommand{\lra}{\longrightarrow}
\newcommand{\doublelra}{\ {\overset{\scr\lra}{\scr\lra}}\ }
\newcommand{\nat}{\natural}
\newcommand{\notmid}{\mkern-5mu\not\mkern5mu\mid}
\newcommand{\Optoc}{\text{\rm Opt}(O_{c^2d},O_B)}
\newcommand{\psim}{\psi^{-}}
\newcommand{\qeq}{\ \overset{??}{=}\ }
\newcommand{\sh}{\sharp}
\newcommand{\thCH}{\theta^{\text{\rm ar}}}
\newcommand{\wht}{\widehat{\theta}}     %%% a replacement
\newcommand{\triv}{1\!\!1}
\renewcommand{\tt}{\otimes}
\newcommand{\und}[1]{\underline{#1}}
\newcommand{\z}{z}  %%% symbol used for the central sign 

\newcommand{\thMW}{\theta^{\text{\rm ar}}}
\newcommand{\tph}{\widetilde{\widehat\phi_1}}
\newcommand{\Pet}{\text{\rm Pet}}

%%%%% from Chapter IX

%%\newcommand{\Gt}{\widetilde{G}}    %%%%  tilde's removed 6/20/04
%\newcommand{\Gt}{G}
%\newcommand{\Ph}{\Phi}
%\newcommand{\pht}{\widetilde{\phi}}
%%\newcommand{\Pht}{\widetilde{\Phi}}%%%%  tilde's removed 6/20/04
%\newcommand{\Pht}{\Phi}
%%\newcommand{\Pt}{\widetilde{P}}
%\newcommand{\Pt}{P}                 %%%%  tilde's removed 6/20/04
%%\newcommand{\Kt}{\widetilde{K}}    %%%%  tilde's removed 6/20/04
%\newcommand{\Kt}{K}
%%\newcommand{\It}{\widetilde{I}}    %%%%  tilde's removed 6/20/04
%\newcommand{\It}{I}
%\newcommand{\Jt}{\widetilde{J}}
%\newcommand{\lt}{\widetilde{\l}}
%\newcommand{\vp}{\varpi}

%%%%%%%%%%%

%%%%  hacking

\newcommand{\thing}{ \raisebox{-6.4pt}{$\tilde{\tilde{}}$}  }   %%% some hacking from 4/14/04
\newcommand{\smallthing}{ \raisebox{-4.4pt}{$\scr\tilde{\tilde{}}$}  }
\newcommand{\ttilde}[1]{\overset{\smash{\thing}}{#1}}
\newcommand{\smallttilde}[1]{\overset{\smash{\smallthing}}{#1}}
\newcommand{\downhookarrow}{\hbox{$\downarrow\hskip -6.1pt\raisebox{6pt}{$\cap$}$}}

%%%% general formating

%\newcommand{\bysame}{\makebox[1.2cm][s]{\hrulefill ,\ }}   %%%%  \bysame is defined already
%\newcommand{\bysame}{$\underline{\text{\hbox to.5in{}}}$}   %%% was used in Textures
\providecommand{\bysame}{\makebox[3em]{\hrulefill}\thinspace}   %%% a fix:  cf. p 322 of Graetzer
\newcommand{\hfb}{\hfill\break}
\newcommand{\margincom}[1]{\marginpar{\bf\raggedright #1}}
\newcommand{\Sec}{\S}

%%%%%%%%%%%%%%

\numberwithin{equation}{section}
\setcounter{section}{0}
\setcounter{MaxMatrixCols}{15}

%%%%%%%%%%%%%%

\newtheorem{theo}{Theorem}[section]
\newtheorem{lem}[theo]{Lemma}
\newtheorem{prop}[theo]{Proposition}
\newtheorem{cor}[theo]{Corollary}
\newtheorem*{main}{Main Theorem}
\newtheorem*{atheo}{Theorem A}
\newtheorem*{btheo}{Theorem B}
\newtheorem{conj}[theo]{Conjecture}
\newtheorem{rem}[theo]{Remark}      %%% seems not to exist in compositio.cls ???
\newtheorem{defn}[theo]{Definition}

\newcommand{\EE}{{\mathbb E}}

\newcommand{\OO}{\text{\rm O}}
\newcommand{\UU}{\text{\rm U}}

\newcommand{\OK}{O_{\smallkay}}
\newcommand{\DI}{\mathcal D^{-1}}

\newcommand{\pre}{\text{\rm pre}}

\newcommand{\Bor}{\text{\rm Bor}}
\newcommand{\Rel}{\text{\rm Rel}}
\newcommand{\rel}{\text{\rm rel}}
\newcommand{\Res}{\text{\rm Res}}
\newcommand{\TG}{\widetilde{G}}

\newcommand{\OL}{O_{\Lambda}}
\newcommand{\OLB}{O_{\Lambda,B}}

\newcommand{\p}{\varpi}

\newcommand{\cutter}{\vskip .1in\hrule\vskip .1in}

\parindent=0pt
\parskip=10pt
\baselineskip=14pt

\newcommand{\PP}{\mathcal P}
\renewcommand{\OO}{\mathcal O}
\newcommand{\BB}{\mathbb B}
\newcommand{\OBB}{O_{\BB}}
\newcommand{\Max}{\text{\rm Max}}
\newcommand{\Opt}{\text{\rm Opt}}
\newcommand{\OH}{O_H}

\newcommand{\phhat}{\widehat{\phi}}
\newcommand{\thetahat}{\widehat{\theta}}

\newcommand{\lbold}{\text{\boldmath$\l$\unboldmath}}
\newcommand{\abold}{\text{\boldmath$a$\unboldmath}}
\newcommand{\cbold}{\text{\boldmath$c$\unboldmath}}
\newcommand{\aabold}{\text{\boldmath$\a$\unboldmath}}
\newcommand{\gbold}{\text{\boldmath$g$\unboldmath}}
\newcommand{\obold}{\text{\boldmath$\o$\unboldmath}}
\newcommand{\fbold}{\text{\boldmath$f$\unboldmath}}
\newcommand{\rbold}{\text{\boldmath$r$\unboldmath}}
\newcommand{\ffbold}{\und{\fbold}}

\newcommand{\deltaBB}{\delta_{\BB}}%{\delta\!\!\!\delta}
\newcommand{\kappaBB}{\kappa_{\BB}}%{\kappa\!\!\!\kappa}
\newcommand{\aboldBB}{\abold_{\BB}}
\newcommand{\lboldBB}{\lbold_{\BB}}
\newcommand{\gboldBB}{\gbold_{\BB}}
\newcommand{\bbold}{\text{\boldmath$\b$\unboldmath}}

\newcommand{\fff}{\phi}

\newcommand{\spp}{\text{\rm sp}}

\newcommand{\pob}{\mathfrak p_{\bold o}}
\newcommand{\kob}{\mathfrak k_{\bold o}}
\newcommand{\gob}{\mathfrak g_{\bold o}}
\newcommand{\pobp}{\mathfrak p_{\bold o +}}
\newcommand{\pobm}{\mathfrak p_{\bold o -}}

%%%%%%%%%%%   saved from old Green function document

\newcommand{\bb}{\frak b}

\newcommand{\bbbold}{\text{\boldmath$b$\unboldmath}}

\renewcommand{\ll}{\,\frak l}
\newcommand{\uC}{\underline{\Cal C}}
\newcommand{\uZZ}{\underline{\ZZ}}
\newcommand{\B}{\mathbb B}
\newcommand{\CL}{\text{\rm Cl}}

\newcommand{\pp}{\frak p}

\newcommand{\OKp}{O_{\smallkay,p}}

\renewcommand{\top}{\text{\rm top}}

\newcommand{\bF}{\bar{\mathbb F}_p}

%%%%%%
\newcommand{\beq}{\begin{equation}}
\newcommand{\eeq}{\end{equation}}

%%%%%%

\newcommand{\Dl}{\Delta(\l)}
\newcommand{\mm}{{\bold m}}

\newcommand{\FD}{\text{\rm FD}}
\newcommand{\LDS}{\text{\rm LDS}}

\newcommand{\dcM}{\dot{\Cal M}}
\newcommand{\bpm}{\begin{pmatrix}}
\newcommand{\epm}{\end{pmatrix}}

\newcommand{\GW}{\text{\rm GW}}

\newcommand{\uk}{\bold k}
\newcommand{\uo}{\text{\boldmath$\o$\unboldmath}}

\newcommand{\uz}{\und{\zeta}}
\newcommand{\duz}{\und{\dot\zeta}}
\newcommand{\ub}{\und{b}}
\newcommand{\uB}{\und{B}}
\renewcommand{\uC}{\und{C}}
\newcommand{\xp}{x_+}
\newcommand{\xm}{x_-}
\newcommand{\tu}{\tilde u}
\newcommand{\sspan}{\text{\rm span}}

\newcommand{\kk}{\mathfrak k}
\newcommand{\tio}{\tilde\o}
\newcommand{\Fock}{\text{\rm Fock}}
\newcommand{\AO}{{\bf AO}}
\newcommand{\qq}{{\bold q}}
\newcommand{\Ps}{\Psi}

\newcommand{\dbs}[1]{\frac{\d \uB_{#1}}{\d s}}
\newcommand{\dbt}[1]{\frac{\d \uB_{#1}}{\d t}}

\newcommand{\OFD}{\text{\rm OFD}}
\newcommand{\Erf}{\text{\rm Erf}}
\newcommand{\Erfc}{\text{\rm Erfc}}
\newcommand{\Arctan}{\text{\rm Arctan}}

\newcommand{\CC}{\mathcal C}
\newcommand{\bC}{C}
%\newcommand{\bC}{\bC}

%%%%%%% these a re defined to allow a convenient shift in notation if desired
\newcommand{\vhy}{V_y}
\newcommand{\dhy}{D'_y}
\newcommand{\Bigwedge}{\mathord{\adjustbox{valign=B,totalheight=.6\baselineskip}{$\bigwedge$}}}

\newcommand{\Psq}{\P^{\square}}
\newcommand{\Ptri}{\P^{\triangle}}

\newcommand{\now}{\count0=\time 
\divide\count0 by 60
\count1=\count0
\multiply\count1 by 60
\count2= \time
\advance\count2 by -\count1
\the\count0:\the\count2}

%\centerline{\it\hfill\today:\ \now}

%\newtheorem*{atheo}{Main Theorem}

%\begin{document}

\title{Theta integrals and generalized error functions, II}

\author{Jens Funke  and  Stephen Kudla}

\address{Department of Mathematical Sciences, Durham University,
South Road, Durham, DH1 3LE, UK}
\email{jens.funke@durham.ac.uk}
\address{
Department of Mathematics,
University of Toronto,
40 St. George St.,
Toronto, ON M5S 2E4
Canada}

\email{skudla@math.toronto.edu}

\maketitle

\section{Introduction}  
The theory of theta series attached to integral lattices $L$ in rational quadratic spaces $L\tt_\Z\Q$ with bilinear form $(\ ,\ )$ of signature $(p,q)$, $pq>0$, 
has a long history including fundamental work of Hecke, Siegel, Maass,  and others. Siegel constructed theta series 
for such indefinite lattices by using majorants and hence obtained functions depending on both an elliptic modular variable $\tau$ 
and a point $z\in D$, the space of oriented negative $q$-planes in $V = L\tt_\Z\R$. These Siegel theta series 
have weight $\frac{p-q}2$ in $\tau$, but, unlike the classical theta series for positive definite lattices, 
they are non-holomorphic. 
In joint work of the second author and John Millson, \cite{KM.I}, \cite{KM.II}, and \cite{KM.IHES}, 
a family of theta series valued in closed differential forms on $D$ was constructed; we will refer to these as theta forms. 
The series obtained by passing to classes in the cohomology of  
the locally symmetric space $\Gamma\back D$, where $\Gamma$ is a subgroup of finite index in the 
isometry group of $L$, were shown to be holomorphic modular forms of weight $\frac{p+q}2$ valued in $H^q( \Gamma\back D)$. 

The resulting theory provides one analogue of the classical holomorphic theta series in the indefinite case.
However, it is still an attractive
challenge to define theta series for indefinite lattices more directly by restricting the summation to lattice vectors in suitable subsets $\mathcal W$ of 
$V$ where the quadratic form 
is positive so that the series
\beq\label{hol-indef-series}
\sum_{x\in h+L}  \P(x;\mathcal W)\, \qq^{Q(x)}, \qquad  \qquad \qq= e(\tau) = e^{2\pi i \tau},\ Q(x) = \frac12(x,x),
\eeq
is termwise absolutely convergence and hence defines a holomorphic function of $\tau$.  Here $\P(\cdot,\mathcal W)$ is 
supported on $\mathcal W$, perhaps valued in $\pm1$.  Unfortunately, such series are typically not modular. 

In his thesis, Zwegers \cite{zwegers} introduced a series of this type for $V$ of signature $(m-1,1)$, where 
$$\P(x;\mathcal W) = \frac12\,(\sgn(x,C') - \sgn(x,C)\,), $$
for $C$ and $C'\in V$ negative vectors in the same component of the cone of negative vectors. 
He showed that the resulting holomorphic series is not modular in general, but that it can be competed to a (non-holomorphic) 
modular form of weight $\frac{m}2$ by adding a suitable series constructed using the error function. 

Recently, Alexandrov, Banerjee, Manschot and Pioline, \cite{ABMP}, proposed a generalization of Zwegers' construction to the 
case of arbitrary signature $(m-q,q)$ where $\P(x;\mathcal W) = \Psq_q(x;\CC)$ is given by 
\beq\label{def-P-sq}
\Psq_q(x;\CC) = 2^{-q}\prod_{j=1}^q (\,\sgn(x,C'_j) - \sgn(x,C_j)\,),
\eeq
for a collection 
$$\CC=\CC^{\square}=\{\{C_1,C'_1\},\{C_2,C'_2\},\dots, \{C_q,C'_q\}\},$$
of pairs of negative vectors satisfying certain incidence relations.  They introduced generalized error functions 
and, in the case $q=2$, used them to construct a (non-holomorphic) modular 
completion of the series (\ref{hol-indef-series}).   Shortly thereafter, Nazaroglu \cite{nazar} handled the case of general signature
along the lines suggested in \cite{ABMP}.  In both \cite{ABMP} and \cite{nazar}, the modularity of the 
non-holomorphic completion is established by using a result of Vigneras, \cite{vigneras}, which asserts the modularity of 
theta like series built from a certain class of functions.  The essential step is to show that suitable combinations of 
generalized error functions define functions in this class and, at the same time, are suitably linked to the 
function $\Psq_q(\cdot,\CC)$. 
Sums of lattice vectors in more general positive polyhedral cones were considered by Westerholt-Raum, 
\cite{raum}; he again uses Vigneras to deduce modularity and also discusses the degenerate case where edges of the cone are allowed to be rational isotropic vectors.

In this paper, we show that the indefinite theta series 
of \cite{zwegers}, \cite{ABMP} and \cite{nazar} can be obtained by integrating the theta forms for $V$ of signature $(p,q)$ 
over certain singular $q$-cubes determined by a collection $\CC$ which is in `good position'. As indicated by the title, this paper is a sequel to 
\cite{kudla.thetaint} where such a result is proved for the case $q=2$. 
We also consider the analogous integrals over singular simplices, where the input data is now a collection 
$\CC = \CC^{\triangle}= \{C_0,C_1,\dots, C_q\}$ of negative vectors in $V$ in `good position'. In particular, any $q$ of them span a negative $q$ plane in $V$ and these
$q$-planes give 
the vertices of a singular simplex in $D$.

To state the results more precisely, we need some notation. 
Let $L$ be an even integral lattice in $V$ with dual lattice $L^\vee$. 
For $\tau = u+iv\in \H$ and $\mu\in L^\vee/L$, the theta form is the closed $\Gamma_L$-invariant  $q$-form on $D$ given by 
$$\theta_\mu(\tau,\ph_{KM}) =  \sum_{x\in \mu+L} \ph_{KM}(\tau,x).$$
Here the Schwartz form
$$\ph_{KM}(\tau,x) = v^{-\frac{p+q}4}\,(\o(g'_\tau)\ph_{KM})(x).$$
is obtained by the action $\o(g'_\tau)$ of the Weil representation on the basic Schwartz form $\ph_{KM}(x)$, cf, section~\ref{section2.2}.
A precise formula for $\ph_{KM}(x)$ is given in section~\ref{section5}. %% and, for convenience, a sketch of the proof of its basic properties is given in Appendix II.  

First consider the `cubical' case. For a collection $\CC=\CC^{\square}$ of $q$ pairs of negative vectors, we can define a $q$-tuple of vectors 
$$B(s) = [(1-s_1)C_1+s_1 C_1', \dots, (1-s_q)C_q+ s_q C_q'] \in V^q,$$
for each $s= [s_1, \dots, s_q]\in [0,1]^q$.  We say that $\CC$ is in good position
if the collection $B(s)$ spans a negative $q$-plane for all $s\in [0,1]^q$. If $\CC$ is in good position, we obtain an oriented 
singular $q$-cube 
$$\phi_\CC: [0,1]^q \lra D,\quad s\mapsto \sspan\{B_1(s_1), \dots, B_q(s_q)\}_{\text{p.o.}},$$
where the subscript `p.o.' indicates that the given $q$-tuple defines the orientation.
Let $S^{\square}(\CC)$ be the resulting singular $q$-cube. 

Next consider the simplicial case. In this case, we suppose that the set of vectors $\CC=\CC^{\triangle}$ is linearly 
independent over $\R$ and that any $q$ of them span a negative $q$-plane. 
Their span $U$ is an oriented $q+1$-plane of signature $(1,q)$ and the dual basis $\CC^\vee = \{C_0^\vee, \dots, C_q^\vee\}$ consists of positive vectors. 
We say that $\CC$ is in good position if, for all 
$$s=[s_0, \dots, s_q]\in \Delta^q= \{\,s\in [0,1]^{q+1}\mid \sum_{i=0}^q s_i=1\,\},$$
the vector 
$$C^\vee(s) = \sum_i s_i C^\vee_i$$ 
is positive. For example, it suffices to require that all entries of the Gram matrix $((C^\vee_i,C^\vee_j))$ are non-negative\footnote{This was pointed out to the 
second author by Sanders Zwegers at the Dublin Conference in June, 2017.}.  For $\CC$ in good position, we obtain a map
$$\phi_\CC: \Delta^q \lra D,\qquad s\mapsto C^\vee(s)^\perp,$$
where the $\perp$ is taken in $U$ and the orientation of $\phi_\CC(s)$ is determined by the normal vector $C^\vee(s)$. 
We write $S(\CC)$ for the resulting singular simplex. We also define
\beq\label{def-phi-tri}
\P_q^{\triangle}(x,\CC) = 2^{-q-1}\bigg(\prod_{j=0}^q(1-\sgn(x,C_j)) + (-1)^q\prod_{j=0}^q(1+\sgn(x,C_j)\ \bigg).
\eeq

We consider the theta integrals
\beq\label{theta-int-def-intro}
I_\mu(\tau,\CC) = \int_{S(\CC)} \theta_\mu(\tau, \ph_{KM}).
\eeq
Note that, by construction,  $I_\mu(\tau,\CC)$ is a (typically non-holomorphic) modular form of weight $\frac{p+q}2$ with transformation law inherited from that of the 
theta form. 

For $1\le r\le q$ and for a collection of vectors $\cbold=\{c_1,\dots,c_r\}$ spanning an oriented negative $r$-plane, let
$E_r(\cbold, x)$, $x\in V$,  be the generalized error function defined by (\ref{def.gen.error}). 
Finally, for $x\in V$, $x\ne0$, let 
$$D_x= \{ z\in D\mid x\perp z\ \},$$
and note that, if $Q(x)>0$, then $D_x$ is a totally geodesic subsymmetric space in $D$ of codimension $q$. Otherwise, $D_x$ is empty. 

Our main result is then the following.  

\begin{main} Assume that $\CC$ is in good position and let $\P_q(x,\CC)$ be $\Psq_q(x,\CC)$ (resp. $\Ptri_q(x,\CC)$)
is the cubical (resp. simplicial) case.  \hfb
(i) The series 
\beq\label{hol-series}
\sum_{x\in \mu+L} \P_q(x,\CC)\,\qq^{Q(x)}
\eeq
is termwise absolutely convergent. \hfb 
(ii) If $\P_q(x,\CC)\ne 0$, then 
$$D_x\cap S(\CC)  = \phi_{\CC}(s(x))$$ 
for a unique point $s(x)\in [0,1]^q$ (resp. $\Delta^q$), the map $\phi_\CC$ is immersive at $s(x)$, and 
$$\P_q(x,\CC) = I(S(\CC),D_x)$$ 
is the intersection number\footnote{If $s(x)$ is on the boundary of $[0,1]^q$, this quantity is defined in 
(\ref{def-inter}) in section~\ref{append-2}.} of $S(\CC)$ and $D_x$ at $\phi_\CC(s(x))$.  \hfb  
(iii) In the cubical case, the theta integral is given explicitly by  
\beq\label{theta-int-formula}
I_\mu(\tau,\CC)  = \sum_{x\in \mu+L} 
(-1)^q\,2^{-q}\sum_I (-1)^{|I|}\, E_q(\bC^I;x\sqrt{2 v})\,\qq^{Q(x)},
\eeq
where for a subset $I\subset \{1,\dots,q\}$, 
$\bC^I$ is the $q$-tuple with $\bC^I_j=C_j$ if $j\notin I$ and $\bC^I_j=C'_{j}$ if $j\in I$, ordered by the index $j$.\hfb
Moreover, $I_\mu(\tau,\CC)$ is the modular completion of the series (\ref{hol-series}). \hfb
(iii) In the simplicial case, the theta integral is given by 
\beq\label{theta-int-formula-tri}
I_\mu(\tau,\CC)  = \sum_{x\in \mu+L} (-1)^q 2^{-q}\sum_{r=0}^{[q/2]}\sum_{\substack{I \\ |I| = 2r+1}} E_{q-2r}(\CC^{(I)}; x\sqrt{2 v})\,\qq^{Q(x)}.
\eeq
where, for a subset $I\subset \{0, 1, \dots, q\}$, let
$\CC^{(I)}$ be the collection of $q+1-|I|$ elements where the $C_i$ with $i\in I$ have been omitted. 
Here $E_0(\dots)=1$. 
\end{main}

\begin{rem}
(1) The series on the right side of (\ref{theta-int-formula}) coincides with that in \cite{ABMP} and \cite{nazar}, at least when 
the collection $\CC$ satisfies their incidence conditions. The incidence conditions they impose on $\CC$, i.e., conditions expressed as requirements on the entries of the 
Gram matrix of $\CC$, imply that $\CC$ is in good position.  On the other hand, the `good position' condition, 
which is a condition on the Gram matrix of the collection $B(s)$ for all $s\in [0,1]^q$,  is sufficient for our results. We leave aside the, perhaps 
subtle, problem of expressing this condition in terms of incidence.  \hfb
(2) Part (ii) of the theorem provides a geometric interpretation of the coefficients 
of the holomorphic generating series as intersection numbers. It would be interesting to see if this interpretation has any significance in the 
physics context which was the original motivation for \cite{ABMP}. \hfb
(3) The proof of (i) is already given in the general case in \cite{kudla.thetaint}.  That the right side of (\ref{theta-int-formula}) is the modular completion of the
series (\ref{hol-series}) is, of course, a main result of \cite{zwegers}, \cite{ABMP}, and \cite{nazar}. \hfb
(4) It is interesting that generalized error functions for negative $r$-planes with $r<q$ occur in the explicit formula in the simplicial case. 
This phenomenon was pointed out by Westerholt-Raum for more general cones, \cite{raum}. 
The indefinite theta series associated to collections $\CC^{\triangle}$ were also discussed by Zwegers in his talk at the 
Dublin conference on Indefinite Theta Functions in June 2017.
\end{rem}

Since the theta integral (\ref{theta-int-def-intro}) can be computed termwise, 
the formulas of parts (iii) and (iv) follow immediately from the formulas for the integral of $\ph_{KM}(x)$ over $S(\CC)$ given in Theorem~\ref{main.theo.1} 
and Theorem~\ref{tetra-theo} respectively.  
These results are, in turn, proved by induction on $q$, where the case $q=1$ is an elementary calculation. 
The key points are the following. First note that both sides of the identities in 
Theorem~\ref{main.theo.1} and Theorem~\ref{tetra-theo} are smooth functions of $x$ and $\CC$, 
so that it suffices to consider the case where $x$ is regular with respect to $\CC$, i.e., where $(x,C)\ne 0$ for all $C\in \CC$. 
As already noted in \cite{FK-I}, the Schwartz form $\ph_{KM}(x)$
comes equipped with an explicit primitive $\Psi(x)$, defined on the set $D-D_x$.  Taking care of the possible singularity, which under the regularity 
assumption occurs at most at a unique interior point of $S(\CC)$, we can apply Stokes' theorem. 
The boundary of $S(\CC)$ consists of singular $(q-1)$-cubes (resp. simplices) in totally geodesic subsymmetric spaces of the form 
$$D'_y = \{ \ z\in D\mid y\in z\,\}$$
for $y= C_j$ or $C'_j$ in $\CC$. Note that $D'_y$ will then be isomorphic to the space of oriented negative $(q-1)$-planes in the space $V_y = y^{\perp}$, of signature
$(p,q-1)$. 
Now the crucial (and remarkable!) fact is that the pullback of the primitive $\Psi(x)$ to such a subspace $D'_y$ 
can be written as an integral transform of the Schwartz $(q-1)$-form $\ph_{KM}^{V_y}(\pr_{V_y}x)$ for $V_y$, cf. Proposition~\ref{lem1.5}. 
By induction, we obtain an expression for the boundary integral as a sum of the corresponding signature $(p,q-1)$ theta integrals. 
Finally, we invoke an inductive identity for generalized error functions from \cite{nazar}, Proposition~\ref{key.nazar}, to conclude the proof. 

\begin{rem} 
(1) One can consider the theta integral $I(\tau,S)$ over any oriented $q$-chain $S$ in $D$, and, if $S$ is compact, this can again be 
computed termwise.  If, moreover, the boundary of $S$ consists of $(q-1)$ chains lying in $D'_y$'s, one can proceed by induction. 
In particular, our result gives an explicit formula for any $q$-chain written as a sum of simplices of the form $S(\CC^{\triangle})$. 
Moreover, since the theta forms are $\Gamma_L$-invariant, their integrals over $\Gamma_L$ equivalent $q$-chains coincide.  \hfb
(2) We can also consider the theta integral $I(\tau,\CC)$ in the degenerate case, when some of the elements in $\CC$ are rational isotropic vectors. Geometrically, this amounts to the $q$-chain $S(\CC)$ going out to some of the rational cusps (of the arithmetic quotient) of $D$. However, while the theta integral over the non-compact region $S(\CC)$ still is convergent by the results of \cite{FMres} (for signature $(m-1,1)$, see \cite{FM.I}), it is in general no longer termwise absolutely convergent (unless one imposes a ``non-singularity" condition as in \cite{K1981}, see also \cite{raum}). One interesting example is signature $(1,2)$, where one can realize the fundamental domain for $\SL_2(\Z)$ as a surface $S(\CC)$ for a certain $\CC$, and the associated theta integral $I(\tau,\CC)$ gives Zagier's non-holomorphic Eisenstein series of weight $3/2$, see \cite{F-thesis,BF2}. 
%
%The most natural such case is that of a simplex associated to a collection of linearly independent negative vectors $\{C_0,C_1,\dots,C_q\}$ such that 
%any $q$ of them span a negative $q$-plane. The indefinite theta series associated to such collections were discussed by Zwegers in his talk at the 
%Dublin conference on Indefinite Theta Functions in June 2017. We plan to provide a description of these as integrals of theta forms in a companion to this paper. 
\hfb
(3) In the companion paper \cite{FK-I}, we consider the theta integral $\int_D  \eta \wedge \theta_\mu(\tau, \ph_{KM})$ against a compactly supported $(p-1)q$ differential form $\eta$ on $D$. In particular, we establish the properties of the primitive $\Psi(x)$ as a current on $D$. 
\end{rem}

Our construction yields a formula for the image of the (typically non-holomorphic) modular form $I_\mu(\tau,\CC)$ under the lowering operator 
$-2i v^2 \frac{\d}{\d \bar \tau}$ or, alternatively, for its shadow given by taking the complex conjugate of this. This formula implies the following, 
cf. section~\ref{section8}.

\begin{cor}\label{cor1.3}  Suppose that $\CC$ is rational collection, i.e., that $C\in L\tt_\Z\Q$ for all $C\in \CC$. Then the shadow of 
$I_\mu(\tau,\CC)$ is a linear combination of products of unary theta series of weight $\frac32$ and complex conjugates of 
indefinite theta series for the spaces $V_C = C^\perp$ for $C\in \CC$. 
\end{cor} 

Here is an outline of the contents of the various sections. Section 2 contains an overview of the construction of theta forms, their 
modular transformation properties, and their relation to geodesic cycles. There is considerable overlap with the material in 
\cite{FK-I}, although our notation and perspective here differs somewhat. Section 3 explains the singular $q$-cubes associated to 
collections $\CC$ in good position and their intersection with the cycles $D_x$ in the regular case. It should be noted that the role of the 
symmetric space $D$ and the singular $q$-cubes is not so evident in \cite{ABMP} and \cite{nazar}.  The use of the `good position' condition
streamlines the treatment, although the important problem of finding equivalent incidence relations is left open. The explicit formula for the 
`cubical' integrals of $\ph_{KM}(x)$ is given in Theorem~\ref{main.theo.1} of Section 4.  In Section 5, we give a more detailed discussion of the
Schwartz forms $\ph_{KM}$ and their primitives. In Section 6  we prove the key formulas for the pullbacks of these forms to the 
spaces $D'_y$.  Section 7 contains the proof of Theorem~\ref{main.theo.1}.  Section 8 contains the computation of the shadows. 
Section 9 contains the analogous computations in the simplicial case, where the are several crucial and interesting differences. 
Some technical details are provided in Appendix I. 
%Appendix II gives a sketch of the construction of the forms $\ph_{KM}$ via the 
%Howe operators and of the proofs of some of their key properties. It is included for convenient reference and to make the paper more 
%self contained. 

\subsection{Thanks}  The second author benefited from the Banff workshop on Modular forms in String Theory in September 2016, 
as well as from discussions with B. Pioline and S. Zwegers 
at the conference, Indefinite Theta Functions and Applications in Physics and Geometry, at Trinity College, Dublin in June of 2017.

\subsection{Notation}

For vectors $x$ and $y$ in a non-degenerate inner product space $V$, $(\ ,\ )$ with $Q(y)\ne 0$, we write
$$x_{\perp y} = x - \frac{(x,y)}{(y,y)} y.$$
Note that 
$$(x_{\perp y}, x'_{\perp y}) = (x,x') -\frac{(x,y)(x',y)}{(y,y)}.$$

We write $e(x) = e^{2\pi i x}$. 

\section{Theta forms and their integrals}
\subsection{Preliminaries}
We begin by reviewing some standard notation and constructions. A good reference is \cite{shintani}.
Suppose that $L$ $(\ ,\ )$ is a lattice of rank $m=p+q$ with an even integral symmetric bilinear form of signature $(p,q)$ with $pq>0$. 
Let $L^\vee\supset L$ be the dual lattice and set $Q(x) = \frac12(x,x)$.
Let $G=O(L\tt_\Z\R)$ be the orthogonal group and let %$\Gamma$ be a subgroup of finite index in
$$\Gamma_L = \{ \ \gamma \in G\mid \gamma L=L, \  \gamma\vert_{L^\vee/L}= \text{\id}\ \}.$$
Let $V = L\tt_\Z\R$ and  let 
$$D =D(V)= \{ \ z\in \text{Gr}_q(V)\mid\  (\ ,\,)\vert_z<0,\, \text{$z$ oriented}\,\}$$
be the space of oriented negative $q$-planes in $V$. 
For $z\in D$, the associated Gaussian is 
$$\ph_0(x,z) = e^{-\pi (x,x)_z},$$
where
$$R(x,z)=-(\pr_z(x),\pr_z(x)),$$
and
$$
(x,x)_z= (x,x)+ 2 R(x,z) 
$$ 
is the majorant determined by $z$. 
%
%
%%\tau&=u+iv\in \H\\
%%\mu&\in L^\vee/L\\
%%\theta_\mu(\tau,z)&=v^{\frac{q}2}\sum_{x\in \mu+L} e^{-2\pi v R(x,z)}\,q^{Q(x)}\\
%%{}&=\text{Siegel's theta function}\\
%%{}&=\text{(non-holomorphic) modular form of weight $\frac{p-q}2$.}
%%\end{align*}
%%As a function of $z\in D$, it is invariant under 
%%$$\Gamma_L = \{\gamma\in O(V)\mid \gamma L=L, \ \gamma\vert_{L^\vee/L}=1\}.$$
%%
%%\section{The Weil representation}
%%
%%A more conceptual picture of this construction is due to Weil.
%%
For fixed $z$, $\ph_0(\cdot, z)=\ph_0(z) \in \mathcal S(V)$ is a Schwartz function on $V$, while, for fixed $x\in V$, $\ph_0(x,\cdot) = \ph_0(x)\in A^0(D)$
is a smooth function on $D$ satisfying the equivariance 
%Since, for  $g\in G= O(V)$,  $R(g x,g z)=R(x,z)$, we have the equivariance 
$$\ph_0(g x,g z) = \ph_0(x,z)$$
for $g\in G$, or equivalently
$$g^*\ph_0(x) = \ph_0(g^{-1}x) =:\o(g)\ph_0(x),$$
where $g^*$ denotes the pullback of functions on $D$ and $\o(g)$ denotes the action of $g$ on $\mathcal S(V)$. 
Thus
\beq\label{Sch-equiv}
\ph_0 \in [\,  \mathcal S(V)\tt A^0(D)\,]^G.
\eeq
%Here
%$$A^0(D)=\text{smooth functions on $D$.}$$

The action $\o$ of $G$ on $\mathcal S(V)$ commutes with the Weil representation action of the 
two-fold cover $G'=\Mp_2(\R)$ of $\SL_2(\R)$ on $\mathcal S(V)$. Recall that for $b\in \R$ and $a\in \R^\times$, there are elements $n'(b)$,  
$m'(a)$, and $w'$ in $G'$ projecting to 
$$n(b) = \bpm 1&u\\{}&1\epm, \quad  m(a)=\bpm a&{}\\{}&a^{-1}\epm, \quad\text{and}\quad w= \bpm {}&1\\-1&{}\epm$$ 
in $\SL_2(\R)$ whose Weil representation action is given by
\begin{align*}
\o(n'(b)) \ph(x) &= e(u Q(x))\,\ph(x)\\
\nass
\o(m'(a))\ph(x) &= |a|^{\frac{m}2}\ph(ax)\\
\nass
\o(w')\ph(x)&= e({\scr\frac{p-q}8})\,\hat{\ph}(x) =e({\scr\frac{p-q}8})\, \int_V \ph(y) \,e(-(x,y))\,dy.
\end{align*}
Then, for $\tau=u+iv\in \H$ and $g'_\tau = n'(u)m'(v^{\frac12})$, we have
$$\o(g'_{\tau})\ph_0(x,z)= v^{\frac{p+q}4} e^{-2\pi v R(x,z)} \,\qq^{Q(x)},\qquad \qq= e(\tau) = e^{2\pi i \tau}.$$

The following invariance property gives rise to the modularity of the theta series. 
Define a vector valued tempered distribution 
$$\Theta_L: \mathcal S(V) \lra \C[L^\vee/L], \qquad \ph\mapsto \Theta(\ph;L) = \sum_{\mu\in L^\vee/L} \theta_\mu(\ph)\,e_\mu,$$
where $e_\mu\in \C[L^\vee/L]$ is the characteristic function of the coset $\mu+L$ and 
$$\theta_\mu(\ph) = \sum_{x\in \mu+L}\ph(x).$$
Let $\Gamma'$ be the inverse image of $\SL_2(\Z)$ in $G'$. Then there is a finite Weil representation $\rho_L$ of $\Gamma'$ acting on 
$\C[L^\vee/L]$, and the theta distribution $\Theta_L$ satisfies
$$\Theta_L(\o(\gamma' )\ph) = \rho_L(\gamma') \Theta_L(\ph).$$

Let $K'$ be the inverse image of $\SO(2)$ in $G'$, and suppose that $\ph$ is eigenfunction of weight $\ell\in \frac12 \Z$ for the Weil representation action of $K'$, i.e., 
$$\o(k'_\theta)\ph = e(\ell \theta)\,\ph, \qquad k_\theta = \bpm \cos \theta&\sin \theta\\-\sin\theta&\cos\theta\epm.$$

%
%$$\vec\Theta$ is invariant under $\Gamma'=\widetilde{\SL_2(\Z)}$, so that, for any $\ph\in \mathcal S(V)$, the function
%$$g'\mapsto \vec\Theta(\o(g')\ph), \qquad  \text{is left $\Gamma'$-invariant.}$$
%
%
%On the other hand, the sum over lattice points defines a tempered distribution:
%$$\ph\ \mapsto\ \Theta_\mu(\ph) = \sum_{x\in \mu+L}\ph(x), \qquad \ph \in \mathcal S(V).$$
%
Then the invariance of the theta distribution together with a standard calculation, \cite{shintani}, pp. 90--98, 
implies that the $\C[L^\vee/L]$-valued theta series 
$$\sum_{\mu\in L^\vee/L} \theta_\mu(\tau,z;\ph)\,e_\mu = v^{-\frac{\ell}{2}} \,\Theta_L(\o(g'_\tau)\ph)$$
is a (non-holomorphic)  vector-valued modular form of weight $\ell$ and type $(\rho_L,\C[L^\vee/L])$. 

The Gaussian is an eigenfunction of $K'$ of weight $\frac{p-q}2$ so that the Siegel theta series 
$\theta_\mu(\tau,z;\ph_0)$ are components of vector valued modular forms and, moreover, via equivariance (\ref{Sch-equiv}), 
are $\Gamma_L$-invariant as functions of $z$, i.e., 
$$\theta_\mu(\tau;\ph_0)\in A^0(D)^{\Gamma_L}.$$
For this semi-classical reformulation of Weil's construction of theta functions we are following Shintani \cite{shintani}, cf. also  \cite{BF}.

%{\bf Remark:}  To get Jacobi forms, you can add an action of the Heisenberg group:
%%The Heisenberg group:
%$$H_\R = \{\ h=[\l,\mu,\kappa]\mid \l, \mu\in \R^n, \kappa\in \R\}$$
%$$(g,h)\cdot(\tau,z) = (g(\tau),(z+\l \tau+\mu)(c\tau+d)^{-1}).$$
%$$\o(h)\ph(x) = e(\kappa) \,e((2x+\l,\mu))\,\ph(x+\lambda).$$
%
%Now, for $\a$ and $\b\in V$,  
%$$\o([\a,\b,0])\o(g'_\tau)\ph_0(x)= e((2x+\a,\b))\,v^{\frac{p+q}4} e^{-2\pi v R(x+\a,z)} \,q^{Q(x+\a)}$$

\subsection{Theta forms}\label{section2.2}

The basic idea is to replace equivariant families of Schwartz functions by equivariant families of {Schwartz forms}, i.e., Schwartz functions 
valued in differential forms on $D$.
Let 
$A^r(D)$ be the space of smooth $r$-forms on $D$.
A main result of \cite{KM.I}, \cite{KM.II} is the explicit construction of a family of Schwartz forms 
\beq\label{KM-equi}
\ph_{KM} \in [\, \mathcal S(V) \tt A^q(D)\,]^G.
\eeq
Thus, for $x\in V$ and $g\in G$, 
$$g^*\ph_{KM}(x)= \ph_{KM}(g^{-1}x) \in A^q(D).$$
In particular, for fixed $x\in V$, $\ph_{KM}(x)$ is a $G_x$-invariant $q$-form on $D$.   For example: $\ph_{KM}(0)$ is a $G$-invariant form.
Under $K'$, 
$$\o(k_\theta)\ph_{KM} = e({\scr\frac{p+q}2}\,\theta)\,\ph_{KM}.$$
Note the shift in weight!  
Moreover, the $q$-form $\ph_{KM}(x)$ is {closed},
$$d\ph_{KM} = 0,$$
where $d: A^q(D)\rightarrow A^{q+1}(D)$ is the exterior derivative. 

Define the {\bf theta form} %($=$ theta function valued in differential forms):
$$\theta_\mu(\tau,\ph_{KM}) := v^{-\frac{\scr (p+q)}4} \theta_\mu(\o(g'_\tau)\ph_{KM}).$$
Then, by construction, 
$\theta_\mu(\tau,\ph_{KM})$ is a closed $\Gamma_L$-invariant $q$-form on $D$ and
hence defines a closed $q$-form on the (orbifold) quotient $M_L=[\Gamma_L\back D]$. 
Moreover, as a function of $\tau$,   $\theta_\mu(\tau,\ph_{KM})$ is a component of a (non-holomorphic) modular form of weight $\frac{p+q}2$
and type $(\rho_L,\C[L^\vee/L])$.

\subsection{Relation to geodesic cycles}

The theta forms define cohomology classes for the locally symmetric space $M_L$ which are related to totally geodesic cycles. 
This was the original motivation for their constuction. 
We recall briefly the basic facts. 
For $x\in V$ with $x\ne 0$, let $V_x = x^\perp$, and let
$$D_x=\{z\in D\mid R(x,z)=0, \text{i.e., $z\subset V_x$}\}.$$
In particular,   $D_x \simeq D(V_x)$ so that $D_x$ is empty if $Q(x)\le 0$, and is a totally geodesic sub-symmetic space of 
codimension $q$ if $Q(x)>0$.

Let 
$\pr_{\Gamma_L}:  D \rightarrow \Gamma_L\back D = M_L,$
and, for $x$ with $Q(x)>0$,  let 
\begin{align*}
Z(x) &= \pr_{\Gamma_L}(D_x),\\
\noalign{a totally geodesic codimension $q$-cycle in $M_L$ with an immersion}
%\nass
i_x:\Gamma_x\back D_x &\lra Z(x) \subset \Gamma\back D.
\end{align*}
Notice that $Z(x)$ depends only on the $\Gamma_L$-orbit of $x$.

The following results are special cases of those obtained in \cite{KM.I}, \cite{KM.II} and \cite{KM.IHES}:
\begin{itemize}
\item[(i)] Suppose that $\eta$ is a closed and compactly supported $(p-1)q$-form on $M_L$. 
Then 
$$\int_{M_L} \eta\wedge \theta_\mu(\tau,\ph_{KM}) = \int_{M_L}\eta\wedge \ph_{KM}(0) + \sum_{\substack{x\in \mu+L\\ Q(x)>0\\ \mod \Gamma_L}}
\bigg(\int_{Z(x)}\eta \,\bigg)\,\qq^{Q(x)}.$$
\item[(ii)] Suppose that $S$ is a compact closed (i.e., $\d S=0$) oriented $q$-cycle on $M_L$. Then 
$$\int_S \theta_\mu(\tau,\ph_{KM})  = \int_{S}\ph_{KM}(0) + \sum_{\substack{x\in \mu+L\\ Q(x)>0\\ \mod \Gamma_L}}  I(S,Z(x))\,\qq^{Q(x)},$$
where $I(S,Z(x))$ is the {intersection number} of the cycles $S$ and $Z(x)$. In particular, both series are termwise absolutely convergent and define 
{\it holomorphic} modular forms of weight $\frac{p+q}2$.
\end{itemize}

Note that these results exactly fit into the framework of \eqref{hol-indef-series}.  Additional discussion is given in \cite{FK-I}.
Many interesting variations are possible! For example, the case of certain non-compact cycles $S$ in $M_L$ is 
considered in joint work of the first author with John Millson, \cite{FM.I}, \cite{FM.II}, \cite{FM.III}.

\subsection{Non-closed compact cycles.}
Suppose that $S$ is a piecewise smooth oriented $q$-chain in the symmetric space $D$. Then, from the general machinery sketched in the previous sections, we obtain 
(non-holomorphic) modular forms, which we will refer to as indefinite theta series, 
\beq\label{indef-theta}
I_\mu(\tau;S) := \int_S \theta_\mu(\tau;\ph_{KM})
\eeq
of weight $\frac{m}2$. 
Since $S$ is compact, we can compute such integrals termwise. Define an operator
\beq\label{IS-op}
I_S: \mathcal S(V) \tt A^q(D) \lra \mathcal S(V),\qquad \ph\mapsto \int_S\ph,
\eeq
from Schwartz forms to Schwartz functions by integrating out the form part. This operator commutes with the Weil representation action of $G'$. 
Thus, we have
\begin{align*}
I_\mu(\tau,S) &= \int_S  v^{-\frac{\scr (p+q)}4} \Theta_\mu(\o(g'_\tau)\ph_{KM})\\
\nass
{}&= v^{-\frac{\scr (p+q)}4} \Theta_\mu(\o(g'_\tau)\big(\,I_S(\ph_{KM})\big),
\end{align*}
so that the indefinite theta series (\ref{indef-theta}) is just the theta series defined by the Schwartz function $I_S(\ph_{KM})$. 
We obtain explicit formulas for the modular forms  $I_\mu(\tau,S)$ whenever we can compute the Schwartz function 
\beq\label{S-int-KM}
I_S(\ph_{KM})\in \mathcal S(V)
\eeq
for a given $q$-chain $S$. 

The remainder of this paper is devoted to the computation in the case of the singular $q$-cubes defined in the next section.
The resulting indefinite theta series are those defined by Zwegers \cite{zwegers} in the case of signature $(m-1,1)$,  by 
Alexandrov, Banerjee, Manschot, and Pioline \cite{ABMP} in the case of signature
$(m-2,2)$ and by Nazarolgu, \cite{nazar}, completing the proposal in \cite{ABMP}, in the case of general signature.

%
%
%Of course, one could more generally consider any compact orientable $q$-cycle $S$ in $D$ and define 
%$I_\mu(\tau,S)$, e.g., a singular simplex or polygon.
%
%
%
%Suppose that 
%$$\phi: [0,1]^q\lra D$$ 
%is a smooth map, i.e., a singular $q$-cube. 
%$$I_\mu(\tau,\phi) := \int_{[0,1]^q} \phi^*(\theta_\mu(\tau,\ph_{KM}))$$
%of weight $\frac{p+q}2$. 
%

\section{Singular $q$-cubes} 

The data $\CC$, cf. (\ref{CC-data}),  introduced in \cite{ABMP} section 6 determines a singular $q$-cube $S(\CC)$ in $D$, whose geometry we discuss in this section. 
We give an explicit formula in terms of generalized error functions for the integral (\ref{S-int-KM}) in the case 
when $S=S(\CC)$ for $\CC$ in `good position'.  However, we have left open the problem of determining explicit conditions on the matrix of inner products of the 
vectors in $\CC$ that are equivalent to good position.

\subsection{The singular $q$-cube $S(\CC)$ and its faces}\label{subsec4.1}

Let
\beq\label{CC-data}
\CC = \{\{C_1,C_{1'}\}, \{C_2,C_{2'}\},\dots,\{C_q,C_{q'}\}\}
\eeq
be a collection of $q$ pairs of negative vectors in $V$.  % satisfying conditions of section 5 of \cite{kudla.thetaint}. 
%Specifically, 
For a subset $I\subset \{1, \dots,q\}$, let $C^I$ be the ordered set $\{C_1^I,\dots,C_q^I\}$ of $q$ vectors where we take $C^I_j= C_j$ if $j\notin I$ and $C^I_j= C_{j'}$ if $j\in I$.
The vectors are ordered according to the index $j$. 
Thus, $C^{\emptyset} = \{C_1,\dots,C_q\}$, etc.  
We would like to have the following `incidence relations':
\begin{itemize}
\item[(Inc-1)]
Each collection $C^I$ spans an oriented negative $q$-plane
$$z^I = \sspan\{C^I\}_{\text{p.o.}}.$$
\item[(Inc-2)] The oriented negative $q$-planes 
$z^I $
all lie on the same component of $D$. 
\end{itemize}
%{\bf Probably additional conditions need to be imposed to insure that we really get a `cube', or, at least so that the singular cube $\rho_\CC$
%is non-degenerate, i.e, that the map $\rho_\CC$ is generically submersive!!}\hfb
These relations, which can be achieved by imposing conditions on the determinants of minors of Gram matrices, should allow us to 
construct a singular $q$-cube with the points $z^I$ as the vertices.  However, as already seen in \cite{kudla.thetaint}, it will be more convenient to 
work with the following formalism.

For $s=[s_1,\dots,s_q]\in [0,1]^q$, 
let
$$B(s) = [B_1(s_1), \dots, B_q(s_q)],$$
where
$$B_j(s_j) = (1-s_j)C_j+s_j C_{j'}.$$
\begin{defn}  A collection $\CC$ is said to be in {\bf good position} if 
for all $s\in [0,1]^q$, 
$$\sspan\{B(s)\}_{\text{p.o.}} = \sspan\{B_1(s_1), \dots, B_q(s_q)\}_{\text{p.o.}}\ \in D.$$
\end{defn}
%We would like to express conditions (i), (ii), and (iii)  in terms of the matrix of inner products
%$$N = \bpm (C_i,C_j)& (C_i, C_{j'})\\ (C_{i'},C_j)& (C_{i'},C_{j'})\epm,$$
%but will postpone this until a later section. 

If $\CC$ is in good position, then relations (Inc-1) and (Inc-2) hold, and we obtain an oriented singular $q$-cube
\begin{equation}\label{S-param}
\rho_\CC:[0,1]^q \lra D, \qquad s = [s_1,\dots,s_q]\mapsto \sspan\{B_1(s_1), \dots, B_q(s_q)\}_{\text{p.o.}}\ \in D.
\end{equation}
with the $z^I$ as its vertices. 
Let $S(\CC)=\rho_\CC([0,1]^q)$ be its image in $D$.  Note that the most degenerate case, in which $C_j = C_{j'}$ for all $j$ and $S(\CC)$ is a point, is allowed.

From now on, unless stated otherwise, we assume that $\CC$ is in good position, so that $\rho_\CC$ and $S(\CC)$ are defined. 

%\begin{rem}
%Note that we have not imposed enough conditions to ensure that this singular cube is non-degenerate. 
%For example, the case in which all of the $z^I $'s coincide is allowed.  Of course, the integrals associated to such degenerate cubes will vanish. 
%\end{rem}

As in \cite{massey}, we define the front  $j$-face
$$\a_j\rho_{\CC}: [0,1]^{q-1} \lra D, \qquad \a_j\rho_{\CC}(s_1,\dots,s_{q-1}) = \rho_{\CC}(s_1, \dots, s_{j-1}, 0,s_j, \dots, s_{q-1}),$$
and back $j$-face
$$\b_j\rho_{\CC}: [0,1]^{q-1} \lra D, \qquad \b_j\rho_{\CC}(s_1,\dots,s_{q-1}) = \rho_{\CC}(s_1, \dots, s_{j-1}, 1,s_j, \dots, s_{q-1}).$$
We write $\d_j^+S(\CC)$ (resp. $\d_j^-S(\CC)$) for the image of $\a_j\rho_{\CC}$ (resp. $\b_j\rho_{\CC}$), viewed as an oriented $(q-1)$-cube. 
With this convention, the boundary of the oriented $q$-cube $S(\CC)$ is given by 
\beq\label{bd-S}
\d S(\CC) = \sum_{j=1}^q (-1)^{j}\big(\d_j^+S(\CC) -\d_j^-S(\CC)).
\eeq

We define collections
\beq\label{C[j]}
\CC[j] =  \{\{C_{1\perp j},C_{1'\perp j}\}, \dots, \widehat{\{C_j,C_{j'}\}},\dots,\{C_{q\perp j},C_{q'\perp j}\}\}
\eeq
and 
\beq\label{C[j']}
\CC[j'] =  \{\{C_{1\perp j'},C_{1'\perp j'}\}, \dots, \widehat{\{C_j,C_{j'}\}},\dots,\{C_{q\perp j'},C_{q'\perp j'}\}\}
\eeq
of $(q-1)$ pairs of negative vectors in $V_j = C_j^\perp$ and $V_{j'}= C_{j'}^\perp$ respectively. 
\begin{lem}  If the collection $\CC$ is in good position for $V$ and $D$, then the collections $\CC[j]$ and $\CC[j']$ are in good position for 
$V_j$, $D(V_j)$ and $V_{j'}$, $D(V_{j'})$ respectively. 
\end{lem}
\begin{proof}
Note that, if we set $s' = [s_1, \dots,  s_{q-1}] \in [0,1]^{q-1}$ and write
$\a_js'= [s_1, \dots, s_{j-1}, 0,s_j, \dots, s_{q-1}]$, then, since $\CC$ is in good position, 
\begin{align*}
\a_j\rho_{\CC}(s') &= \rho_{\CC}(\a_j s')\\
\nass
{}&= \sspan\{B_1(s'_1), \dots, B_{j-1}(s'_{j-1}),C_j, B_{j+1}(s'_j),\dots, B_q(s'_{q-1})\}_{\text{p.o.}}\\
\nass
{}&=\sspan\{B_1(s'_1)_{\perp j}, \dots, B_{j-1}(s'_{j-1})_{\perp j},C_j, B_{j+1}(s'_j)_{\perp j},\dots, B_q(s'_{q-1})_{\perp j}\}_{\text{p.o.}}\in D,
\end{align*}
which implies that $\CC[j]$ is in good position for $V_j$ and $D(V_j)$.  Similarly for $\CC[j']$. 
\end{proof}
We write $S(\CC[j])$ and $S(\CC[j'])$ for the corresponding oriented singular $(q-1)$-cubes in $D(V_j)$ and $D(V_{j'})$ with 
parametrizations analogous to (\ref{S-param}),
$$\rho_{\CC[j]}:[0,1]^{q-1}\lra D(V_j)$$ 
and 
$$\rho_{\CC[j']}:[0,1]^{q-1}\lra D(V_{j'}).$$   

In the notation defined in (\ref{defkappaj-jay}),  we let $\kappa_j = \kappa_{C_j}[j]$ and $\kappa_{j'} = \kappa_{C_{j'}}[j]$ so that 
\beq\label{wall-front}
\kappa_j\circ \rho_{\CC[j]} = \a_j\rho_{\CC},
\eeq
and
\beq\label{wall-back}
\kappa_{j'}\circ \rho_{\CC[j']} = \b_j\rho_{\CC}.
\eeq
Here the key point to note is that 
\begin{align*}
&\,\sspan\{B_1(s_1),\dots,B_q(s_q)\}_{\text{p.o.}}\vert_{s_j=0}\\
\nass
{}=&\, \sspan\{B_{1\perp j}(s_1), \dots , B_{(j-1)\perp j}(s_{j-1}), \uC_j,B_{(j+1)\perp j}(s_{j+1}),\dots,
B_{q\perp j}(s_q)\}_{\text{p.o.}}\\
\nass
{}=&\,\kappa_j\circ \rho_{\CC[j]}(s_1,\dots,\widehat{s_j},\dots,s_q),
\end{align*}
where, for example, 
$$B_{1\perp j}(s_1) = (1-s_1) C_{1\perp j}+s_1 C_{1'\perp j}.$$

\subsection{The regular case}

Following (6.5) of \cite{ABMP}, for a vector $x\in V$, let $\P_q(x,\CC)=\Psq_q(x,\CC)$ be as in (\ref{def-P-sq}), 
%$$\P_q(x,\CC)=\Psq_q(x,\CC)= \frac{1}{2^q}\,[\,\sgn(x,C_1) - \sgn(x,C_{1'})\,]\dots [\,\sgn(x,C_q) - \sgn(x,C_{q'})\,],$$
Recall that $\sgn(0)=0$.

Recall from \cite{kudla.thetaint} that a vector $x\in V$ is said to be {\bf regular} with respect to $\CC$ if $(x,C)\ne0$ for all $C\in \CC$. 
Parts (i) and (ii) of the following are an analogue of Lemma~4.2 in loc. cit. and the proofs given there extend immediately to the general case. 
Part (iii) will be proved in Appendix I, where the definition of the local intersection number will also be reviewed.  
\begin{lem}\label{old-lemma} Let $\CC$ be a collection in good position. \hfb
(i) If $x\in V$ is regular with respect to $\CC$, then $D_x\cap S(\CC)$ is non-empty if and only if $\P_q(x,\CC)\ne0$, and, in this case
$D_x\cap S(\CC)= \rho_\CC(s(x))$ 
for a unique point $s(x)\in (0,1)^q$ given by 
\beq\label{int-point}
s(x)_j = \frac{(x,C_j)}{(x,C_j)-(x,C_{j'})}.
\eeq
(ii) If $x\in V$ is any vector with $\P_q(x,\CC)\ne0$, then $D_x\cap S(\CC)$ consists of a single point $\rho_\CC(s(x))$ with 
$s(x)\in [0,1]^q$ given by (\ref{int-point}).\hfb
(iii)  If $x\in V$ is any vector with $\P_q(x,\CC)\ne0$ and $s(x)$ is as in (ii), then the map $\rho_\CC$ is immersive at $s(x)$, 
and the quantity $\P_q(x)$ is the local intersection number
%\footnote{If $s(x)\in (0,1)^q$ is an interior point and $z_0=\rho_\CC(s(x))$, then this is defined 
%in the usual way in terms of the orientations of $T_{z_0}(D)$, $T_{z_0}(D_x)$, $T_{z_0}(S(\CC))$, where the latter is defined via the immersion.
%If $s(x)$ lies on the boundary of $[0,1]^q$, then, since $D$ is open, the map $\rho_{\CC}$ can be extended to an open neighborhood of 
%$s(x)$ in $\R^q$, again immersive at $s(x)$, and the intersection number is defined by including a weight factor of $2^{-r}$ 
%where $r$ is the number of walls of $[0,1]^q$ containing $s(x)$. For example, the weight factor is $2^{-q}$ if $s(x)$ is a vertex of the cube. } 
of $D_x$ and $S(\CC)$ at $s(x)$. 
A precise definition of this quantity is given in (\ref{def-inter}) in section~\ref{append-2}. 
\end{lem}

Again, as in \cite{kudla.thetaint}, we say that $\CC$ is in {very good position} if it is in good position and $\rho_\CC$ is an embedding, i.e., an injective immersion. 
The following is easily checked by the general version of the calculation in section~6.3 of \cite{kudla.thetaint}.
\begin{lem}\label{embedding-lemma}
If $\CC$ is in good position and the $2^q$ vectors in $\CC$ are linearly independent, then $\CC$ is in very good position. 
\end{lem}
%This observation shows that there is, in general, a large supply of $\CC$'s in good position. 
%\begin{proof} At a point $s\in [0,1]^q$, we have 
%$$(\rho_\CC)_*(\frac{\d}{\d s_j}) =.$$
%\end{proof}

\section{Cubical Integrals and generalized error functions}

In this section, we state our main result, an explicit expression for the Schwartz function (\ref{S-int-KM}) defined by the integral 
$$I(x;\CC) := \int_{S(\CC)}\ph_{KM}(x)$$
of the $q$-form $\ph_{KM}(x)$ over the singular $q$-cube $S(\CC)$ in $D$  in terms of generalized error functions,  as suggested in section 5 of \cite{kudla.thetaint}.
%Since the singular $q$-cube is compact, $I(x;\CC)$ is in fact a scalar valued Schwartz function on $V$. 
%Indeed, the analogous integral over any compact $q$-cell $S$ in $D$ defines a Schwartz function, and it is an interesting question
%to determine this function more explicitly for various $S$. 
%\subsection{Cubical Integrals and generalized error functions}
%We express $I(x;\CC)$ in terms of generalized error functions.
Recall, \cite{ABMP}, (6.1), that, for a collection of vectors negative $\bC=\{C_1,\dots,C_q\}$ spanning an oriented $q$-plane $z\in D$, 
and for $x\in V$, the generalized error function is given by the integral
\beq\label{def.gen.error}
E_q(\{C_1,\dots,C_q\};x)  = \int_{z} e^{\pi (y-\pr_z(x),y-\pr_z(x))}\,\sgn(C_1,y)\,\sgn(C_2,y)\dots\sgn(C_q,y)\,dy,
\eeq
where the measure $dy$ is normalized so that
$$\int_z e^{\pi (y,y)}\,dy =1.$$
We will frequently abbreviate this to $E_q(\bC;x)$ and write
\beq\label{def.sgnCx}
\sgn(\bC;y) = \sgn(C_1,y)\,\sgn(C_2,y)\dots\sgn(C_q,y).
\eeq

Our main result is the following explicit formula for $I(x;\CC)$. % {\it Maybe $2^{-q}$ in front?}
\begin{theo}\label{main.theo.1} Suppose that $\CC$ is in good position. Then 
\beq\label{inductive.conjecture}
I(x;\CC) =  (-1)^q\,2^{-q}\sum_I (-1)^{|I|}\, E_q(\bC^I;x\sqrt{2})\,e^{-\pi (x,x)},
\eeq
where, as in section~\ref{subsec4.1}, for a subset $I\subset \{1,\dots,q\}$, 
$\bC^I$ is the $q$-tuple with $\bC^I_j=C_j$ if $j\notin I$ and $\bC^I_j=C_{j'}$ if $j\in I$, ordered by the index $j$.
%$E_q(z;x)$ is the generalized error function, \cite{ABMP}, (6.1).
%Here $c$ is a constant, depending only on $q$, determined by our normalizations. 
\end{theo}

The $2^q$ terms in the sum on the right side of (\ref{inductive.conjecture}) are generalized error functions associated to the vertices 
$z^I =\sspan\{\bC^I\}_{\text{p.o.}}$ of $S(\CC)$ of the singular $q$-cube evaluated on the projections of $x$ to those $q$-planes. 

\begin{rem}
In the case $q=2$, the expression given in Theorem~\ref{main.theo.1} is the negative of the expression found in \cite{kudla.thetaint}.   But there is a simple explanation, namely the 
orientation of $S(\CC)$ used there is defined by the `loop' in (3.11), but this is the opposite of the orientation we use here, defined by the singular square $\rho_\CC$.
\end{rem}

The proof of Theorem~\ref{main.theo.1} by induction on $q$ is given in section~\ref{proof-Thm-4.1}.  

\section{Review of the Schwartz form $\ph_{KM}$ and its relatives}\label{section5}

%\section{Theta forms}

In this section, we review the basic facts about the Schwartz forms $\ph_{KM}(x)$ which we need. 
%For convenient reference, a more detailed description of the derivation of these facts is given in Appendix II. 

\subsection{Local formulas}\label{local-formulas}
We fix a base point $z_0\in D$ and 
an orthonormal basis $\{e_1,\dots,e_m\}$, $m=p+q$,
$(e_r,e_s) = \e_r \delta_{rs}$, $\e_r =+1$ for $1\le r\le p$ and $\e_r=-1$ for $r>p$, 
with 
$$z_0 = \sspan\{e_{p+1}, \dots, e_m\}_{\text{p.o.}}.$$
In particular
%\footnote{Jens: Going back through \cite{KM.I}, the isomorphism used here is given by
%$$x = \sum_{r=1}^m x_r e_r.$$
%I still have to check the form part.}
$V\simeq \R^m$, and the Gaussian is given by 
\beq\label{fix-coord}
\ph_0(x)=\ph_0(x,z_0) = e^{-\pi \sum_j x_j^2}\ \in \mathcal S(V),\qquad x = \sum_i x_i e_i.
\eeq
%where $S(V)$ is the Schwartz space of $V$. 
Let $K$ be the stabilizer of $z_0$ in $G$ and write $\g_o = \Lie(G) = \kk_o+\pp_o$ 
where $\kk_o=\Lie(K)$ and $\pp_o$ are the $+1$ and $-1$ eigenspace for the Cartan involution at $z_0$. 
There is a canonical isomorphism $T_{z_0}(D) \simeq \pp_0$.
Under the idenitfication 
$$V\tt V \isoarrow \End(V), \quad (v_1\tt v_2)(v) = (v_2,v) v_1,$$
a basis for $\pp_0$ is given by
$$X_{\a\mu} = e_\a\tt e_\mu + e_\mu\tt e_\a, \qquad 1\le \a\le p< \mu\le p+q.$$
Let $\o_{\a\mu}$ be the dual basis for $\pp_o^*$. 

By the equivariance property (\ref{KM-equi}), $\ph_{KM}(x)$ is determined by the element of the complex
$$[\mathcal S(V)\tt \sideset{}{^{\bullet}}{\bbigwedge}(\pp_o^*)]^K$$
obtained by restriction to the point $z_0$. 

For $1\le s, t\le q$, let
$$\o(s) = \sum_{j=1}^p x_j \,\o_{j,p+s}\ \in \pp_o^*,$$
and 
$$\O(s,t) = \sum_{j=1}^p \o_{j,p+s}\wedge \o_{j,p+t}\ \in \sideset{}{^{2}}\bbigwedge(\pp_o^*).$$
For  $\l$ with $0\le \l\le [q/2]$, we define $q$-forms
\beq\label{AOlambda}
\AO_\l(q)=A\big[\,\o(1)\wedge\dots \wedge \o(q-2\l)\wedge \O(q-2\l+1,q-2\l+2)\wedge \dots\wedge \O(q-1,q)\,\big],
\eeq
where 
$A$ is the alternation
\beq\label{alternation}
A\big[\,\o(1)\wedge\dots\wedge \O(t-1,t)\,\big] = \frac{1}{t!} \sum_{\s\in S_t} \sgn(\s)\, \o(\s(1))\wedge\dots \wedge \O(\s(t-1),\s(t)).
\eeq
%%and\footnote{To be precise here, we take the last $2\l$ of the given indices in the $\O$'s and the remaining indices in the $\o$'s, then apply alternation.} 
%$(t-1)$-forms 
%$$\AO_\l(t;s) = \AO_\l(\{1,\dots,\hat s,\dots,t+1\}).$$
Note that these are homogeneous of degree $q-2\l$ in $x$, and it will sometimes be useful to write $\AO_\l(q)(x)$ to indicate this dependence.
With this notation, we have the following formula for the restriction of $\ph_{KM}(x)$ at the point $z_0$, cf. \cite{KM.I} p. 371,
\begin{equation}\label{ph-formula}
\ph_{KM}(x) = 2^{q/2}\sum_{\l=0}^{[q/2]} C(q,\l)\,\AO_\l(q)(x)\,\ph_0(x),
\end{equation}
where 
\beq\label{Ctl}
C(t,\l) = (-\frac{1}{4\pi})^\l\frac{t!}{2^\l \l! (t-2\l)!}.
\eeq

There are two auxiliary $q-1$ forms associated to $\ph_{KM}(x)$ which will play a fundamental role in our calculations. 
We will recall their relation to $\ph_{KM}(x)$ in a moment. 
The first of these
is given by
%\footnote{JENS:  There is a potential sign issue here.  According to the footnote on p.1,  $x_{p+s}$ is given by 
%$-(x,\zeta_s)$ if we are using the o.n. frame $\zeta$ as basis for $z$. We need to check this. {\color{red} I believe that the red version is correct!}}
\beq\label{psi0-formula}
\psi_{KM}(x) = 2^{q/2-1}\sum_{\l=0}^{[(q-1)/2]} \sum_{s=1}^q {(-1)^{s} x_{p+s}}\,C(q-1,\l)\,\AO_\l(q;s)(x)\,\ph_0(x),
\eeq
%{\it Maybe correct:}
%$$\psi_{KM}(x) = 2^{q/2}\,\sum_{\l=0}^{[(q-1)/2]} \sum_{s=1}^q (-1)^{s-1} x_{p+s}\,C(q-1,\l)\,\AO_\l(q;s)\,\ph_0(x),$$
where the $(q-1)$-form $\AO_\l(q;s)$ is defined by the alternation analogous to $\AO_\l(q-1)$ but for the index set 
$\{1, \dots, \hat{s}, \dots, q\}$ replacing $\{1, \dots,q-1\}$. For example, $\AO_\l(q,q) = \AO_\l(q-1)$. 

Now we include the parameter $\tau = u+iv$. 
%For any negative $q$-plane $z$ in $V$ and $x\in V$, let $\pr_z(x)$ be the projection of $x$ to $z$ and let 
%$$R(x,z) = -(\pr_z(x),\pr_z(x)),$$
%and let 
%$$D_x= \{z\in D\mid R(x,z)=0\},$$
%so that $D_x$ is the set of oriented negative $q$-planes orthogonal to $x$. 
%Also, let $z_0$ be the negative $q$-plane spanned by $v_{p+1}, \dots,v_m$.  
%Then
Writing
\begin{align*}
%\ph_0(x) &= \ph_0^0(x)\,e^{-\pi(x,x)},\qquad \ph_0^0(x) = e^{-2\pi R(x,z_0)},\\
%\nass
\ph_{KM}(x) &= \ph_{KM}^0(x)\,e^{-\pi(x,x)},\\
\nass
\psi_{KM}(x)& = \psi^0_{KM}(x)\,e^{-\pi(x,x)},
\end{align*}
we have, for $\qq= e(\tau)$ and $Q(x) = \frac12(x,x)$,  
\begin{align}
%\ph_0(\tau,x) &= v^{\frac{q}2}\, \ph_0^0(v^{\frac12}x)\,\qq^{Q(x)} = v^{-\frac{p-q}4} \o(g'_\tau)\ph_0(x,z_0),\\
%\nass
\ph_{KM}(\tau,x) &= \ph_{KM}^0(v^{\frac12}x)\,\qq^{Q(x)}= v^{-\frac{p+q}4} \o(g'_\tau)\ph_{KM}(x),\\
\noalign{and}
\psi_{KM}(\tau,x) & = %2^{\frac{q}2}\,
v\,\psi_{KM}^0(v^{\frac12}x)\,\qq^{Q(x)}.\label{psi0-to-psi}
\end{align}

Note that 
$$-2i\,\frac{\d}{\d \bar\tau}\ph_{KM}(\tau,x)  = \frac{\d}{\d v}\big\{\ph_{KM}^0(v^{\frac12}x)\big\}\,\qq^{Q(x)}.$$

%\subsection{The primitive $\Psi_{KM}(\tau, x)$.}
On the set of $x$ such that $R(x,z_0)\ne 0$, let
\beq\label{def-Psi-0}
\Ps_{KM}^0(x) = -\int_{1}^\infty \psi^0(t^{\frac12}x)\,t^{-1}\,dt.
\eeq
The point here is that
$$\psi^0(t^{\frac12}x) =\big(\,\text{form valued poly in $t^{\frac12}x$}\,\big) \cdot e^{-2\pi t R(x,z_0)},$$
so that the integral only makes sense when $R(x,z_0)>0$. 
For $x$ with $R(x,z_0)>0$, let
\beq\label{def-Psi}
\Psi_{KM}(\tau,x) := \Psi^0_{KM}(v^{\frac12}x)\,\qq^{Q(x)}=-\int_v^{\infty}\psi^0(t^{\frac12}x)\,t^{-1}\,dt\ \qq^{Q(x)}.
\eeq

The following basic relations between the primitives $\psi_{KM}(\tau,x)$, $\Psi_{KM}(\tau,x)$ and the form $\ph_{KM}(\tau,x)$
are given in   \cite{FK-I}, Section 3, Proposition~3.2.%Appendix II.
\begin{lem}\label{lem2.3} (i) 
$$%\label{john-eq-tau}
-2i v^2 \frac{\d}{\d \bar\tau}\ph_{KM}(\tau,x) = d \psi(\tau,x) = v\,d\psi^0_{KM}(v^{\frac12}x)\,\qq^{Q(x)}.
$$
(ii)
$$d\Psi_{KM}(\tau,x) = \ph_{KM}(\tau,x),\qquad R(x,z_0)\ne 0,$$
and
$$d\Psi_{KM}^0(x) = \ph_{KM}^0(x),\qquad R(x,z_0)\ne 0.$$
\end{lem}
%\begin{proof} Note that the relation of Lemma~\ref{lem1.2} relates operators on the Lie algebra cohomology complex. 
%For the corresponding functions %\footnote{Now the fact that the above expression of $\bar\d$ may need a factor of $-\frac12$ will have to be fixed!} 
%of $\tau$,  this amounts to 
%\begin{equation}\label{john-eq-tau}
%-2i v^2 \frac{\d}{\d \bar\tau}\ph_{KM}(\tau,x) = d \psi(\tau,x) = v\,d\psi^0_{KM}(v^{\frac12}x)\,\qq^{Q(x)}.
%\end{equation}
%But also note that 
%$$-2i v^2 \frac{\d}{\d \bar\tau}\ph_{KM}(\tau,x) = v^2\,\frac{\d}{\d v}\ph_{KM}^0(v^{\frac12}x)\,\qq^{Q(x)},$$
%so that 
%$$d\psi^0_{KM}(v^{\frac12}x) = v\,\frac{\d}{\d v}\ph_{KM}^0(v^{\frac12}x).$$
%Then, we have
%\begin{align*} 
%d\Psi_{KM}(\tau,x) &= -\int_v^{\infty}d\psi_{KM}^0(t^{\frac12}x)\,t^{-1}\,dt\,\qq^{Q(x)}\\
%\nass
%{}&=-\int_v^{\infty}\frac{\d}{\d t}\ph_{KM}^0(t^{\frac12}x)\,dt\,\qq^{Q(x)}\\
%\nass
%{}&= \ph_{KM}^0(v^{\frac12}x)\,\qq^{Q(x)} = \ph_{KM}(\tau,x),
%\end{align*}
%as claimed. 
%\end{proof}

Taking homogeneity in $x$ of various terms into account and writing $R=R(x,z_0)$, we have the explicit formulas
%\begin{align}\label{Psi-formula}
%\Psi_{KM}(\tau,x) &=2^{\frac{q}2}\,2\sum_{\l=0}^{[(q-1)/2]} \sum_{s=1}^q (-1)^{s-1} x_{p+s}\,C(q-1,\l)\,\AO_\l(q;s)\,\\
%\nass
%{}&\qquad\qquad\qquad\times (2\pi R)^{-\frac12(q-2\l)}\,\Gamma(\frac12(q-2\l), 2\pi R v)\,\qq^{Q(x)}.\notag
%\end{align}
%{\it Maybe correct:}
\begin{align}\label{Psi-formula-corrected}
\Psi_{KM}(\tau,x) &= 2^{\frac{q}2-1}\sum_{\l=0}^{[(q-1)/2]} \sum_{s=1}^q  {(-1)^{s-1} x_{p+s}}\,C(q-1,\l)\,\AO_\l(q;s)\,\\
\nass
{}&\qquad\qquad\qquad\times (2\pi R)^{-\frac12(q-2\l)}\,\Gamma\big(\frac12(q-2\l), 2\pi R v\big)\,\qq^{Q(x)}.\notag
\end{align}
and
\begin{equation}\label{ph-formula-tau}
\ph_{KM}(\tau,x) = 2^{q/2}\sum_{\l=0}^{[q/2]} C(q,\l)\,\AO_\l(q)\,v^{\frac12(q-2\l)}\,e^{-2\pi v R}\,\qq^{Q(x)}.
\end{equation}

\subsection{Global formulas}

%For $V$, $(\ ,\ )$ of signature $(p,q)$, as above, let $D$ be the space of oriented negative $q$-planes in $V$. 
We now explain how the 
formulas of the previous section define global differential forms on $D$.  We will use the notation and conventions explained in \cite{kudla.thetaint}, 
especially the Appendix, which we now briefly recall. 

Let
$$\FD = \{\ \zeta = [\zeta_1,\dots,\zeta_q]\in V^q\mid  (\zeta,\zeta) := ((\zeta_i,\zeta_j))  <0\,\},$$
be the bundle of oriented negative frames, and let
$$\OFD = \{ \ \zeta = [\zeta_1,\dots,\zeta_q]\in V^q\mid (\zeta,\zeta) =-1_q\, \},$$
 be the bundle of oriented orthonormal negative frames. Let $\pi: \FD\rightarrow D$ be the natural projection, taking $\zeta$ to its oriented span.
 Then, for $\zeta\in \OFD$, we have an identification of tangent spaces
 $$V^q \simeq T_\zeta(\FD) \supset T_\zeta(\OFD) = \{\,\eta = [\eta_1,\dots,\eta_q] \in V^q \mid (\eta,\zeta) + (\zeta,\eta)=0\,\}.$$
 For $z\in D$, we let $U(z)=z^\perp$. Then the `horizontal' subspace $U(z)^q \subset T_\zeta(\OFD)$ is identified with $T_z(D)$ under $d\pi_\zeta$. 
 Note that, while the space $U(z)^q$ depends only on $z$, the identification with $T_z(D)$ depends on $\zeta$. The identifications for different choices of $\zeta$
 differ by the action of $\SO(q)$. 
 
% It will be convenient to describe differential forms on $D$ in terms of their pullbacks on $\OFD$. 

A priori, the expressions given in (\ref{ph-formula-tau}) and (\ref{Psi-formula-corrected}) are elements of $S(V)\tt \bbigwedge^{r}(\pp_o^*)$
with $r= q$ and $q-1$ respectively,  where $\pp_o$ is identified with the tangent space to $D$ at the base point 
$$z_0=\sspan\{e_{p+1},\dots,e_{p+q}\}_{\text{p.o.}} \in D$$ 
determined by our chosen orthonormal basis. They yield global formulas as follows.  For any $\zeta\in \OFD$, the function $R(x,z)$ is defined by
$R(x,z) = (x,\zeta)(\zeta,x)$.  %= \sum_{s=1}^q (x,\zeta_s)^2.$$
For vectors $\eta = [ \eta_1, \dots, \eta_q]$ and $\mu= [\mu_1,\dots,\mu_q]$ in $U(z)^q$,  
define 
\beq\label{define-forms}
\o(s)(\eta) = (x,\eta_s),\qquad \O(s,t)(\eta,\mu) = (\eta_s,\mu_t)- (\eta_t,\mu_s).
\eeq
Also note that, in the global version of (\ref{Psi-formula-corrected}), 
\beq\label{sign-shift}
x_{p+s} = -(x,\zeta_s).
\eeq

\begin{lem} With these definitions, the $q$-forms $\AO_\l(q)$ and $q-1$-forms $\AO_\l(q;s)$ on $U(z)^q$ are invariant under $\SO(q)$
and hence define forms on $T_{z}(D)$. 
\end{lem}
\begin{proof} We observe that for some non-zero constant $c$, 
$$
\AO_\l(q)(\eta^1,\dots,\eta^q) = c\,\det \begin{pmatrix} (x,\eta_1^1)&\dots&(x,\eta^1_{q-2\l})& \eta^1_{q-2\l+1}&\dots&\eta^1_q\\
\vdots&{}&\vdots&\vdots&{}&\vdots\\
(x,\eta_1^q)&\dots&(x,\eta^q_{q-2\l})& \eta^q_{q-2\l+1}&\dots&\eta^q_q
\end{pmatrix}
$$
where, in expanding the determinant, the product of vectors is taken using $(\ ,\ )$. 
\end{proof} 

Thus (\ref{ph-formula-tau}) (resp.  (\ref{Psi-formula-corrected}))  defines a global $q$ form $\ph_{KM}(\tau,x)$ on $D$ (resp. a global $q-1$-form 
$\Psi_{KM}(\tau,x)$  on $D - D_x$) and these forms satisfy
$$d\Psi_{KM}(\tau,x) = \ph_{KM}(\tau,x)$$ 
on $D-D_x$.

\begin{rem}
The formula for the pullback for these forms to $\OFD$ involves additional terms determined by the requirement that the forms vanish if one of the input tangent 
vectors is vertical, i.e., in the kernel of $d\pi_\zeta$.  We will not need these expressions. 
\end{rem}

%Then the $q$-form $\ph_{KM}^0(x)$ on $D$ has pullback
%$$\pi^*\ph_{KM}^0(x)\vert_{U(z)^q} = $$
%and the $q-1$ form $\Psi_{KM}^0(x)$ on $\{z\in D\mid R(x,z)\ne0\,\}$ has pullback
%$$\pi^*\Psi_{KM}^0(x)\vert_{U(z)^q} = .$$
%%Note that the form $\Psi_{KM}(\tau,x)$ is defined on the set
%%$$\{\, [x,z]\in V\times D\mid R(x,z)\ne0 \,\}.$$
%The full formula for these forms  
%

\section{The pullback to certain sub-symmetric spaces}
%If $C\in V$ is a vector with $(C,C)<0$, the orthogonal complement $V_C:=C^\perp$ has signature $(p,q-1)$, and the subspace of $V$ spanned by $C$ 
%together with a negative $q-1$-plane in $C^\perp$ 
%is a negative $q$-plane in $V$. Thus we obtain a natural embedding $O(p,q-1)\hookrightarrow O(p,q)$ and of the corresponding symmetric spaces. 
%We now refine this picture slightly by taking into account orientations. 

Suppose that $y\in V$ is a negative vector, and let
$$\vhy = y^\perp,$$
$$\dhy = \{\,z\in D\mid y\in z\},$$
and
$$D(\vhy) =  \{ z= \text{oriented neg. $(q-1)$-plane in $\vhy$}\}.$$
For the properly oriented orthogonal frame bundle $\OFD(\vhy)\rightarrow D(\vhy)$,
there is an embedding 
\beq\label{defkappaj}
\kappa_y:  \OFD(\vhy)\hookrightarrow \OFD, \qquad \zeta \mapsto [\und{y},\zeta],
\eeq
where $\und{y} = y |(y,y)|^{-\frac12}$, and a resulting embedding
\beq\label{defkappaj-2}
\kappa_y:D(\vhy)\isoarrow \dhy \subset D.
\eeq

A fundamental result is the following pullback formula, which we find rather striking
%\footnote{There is a sign ambiguity, since I am using $x_{p+1} = -(x,\und{y})$ in the proof.}. 

\begin{prop}
 For $x\in V$, write $x = -(x,\und{y})\,\und{y} + x_{\perp y}$, so that $x_{\perp y}$ is the $\vhy$-component of $x$. Then \hfb
(i)
\begin{align*}
\kappa_y^*\psi_{KM}^0(x) &= 2^{-\frac12}\,(x,\und{y})\, e^{-2\pi (x,\und{y})^2}\,\ph_{KM}^{\vhy,0}(x_{\perp y}).\\
\noalign{\noindent
(ii)}
\kappa_y^*(\psi_{KM}(\tau,x)) &= 2^{-\frac12}\, v^{\frac32}\,(x,\und{y})\,e^{-2\pi v(x,\und{y})^2} \qq^{-\frac12(x,\und{y})^2}\, \ph_{KM}^{\vhy}(\tau,x_{\perp y}).
\end{align*}
Here $ \ph_{KM}^{\vhy,0}(\tau,\cdot)$ is the $\ph_{KM}^0$ Schwartz $(q-1)$-form on $D(\vhy)$.
\end{prop} 
\begin{proof} 
The map on tangent spaces is given by 
$$d\kappa_y:T_{\zeta}(\OFD(\vhy)) \lra T_{\kappa_y(\zeta)}(\OFD), \qquad \eta=[\eta_1,\dots, \eta_{q-1}] \mapsto [0,\eta_1,\dots, \eta_{q-1}],$$
and this map is compatible with the `horizontal' subspaces. 
It follows that any term in $\psi_{KM}^0(x)$ involving an index $s=1$ in the differential form will vanish under pullback.  Thus, by (\ref{psi0-formula}),  we have
\begin{align*}
\kappa_y^*\psi_{KM}^0(x) &= 2^{\frac{q}2-1}\,(x,\und{y})\, e^{-2\pi (x,\und{y})^2}\sum_{\l=0}^{[(q-1)/2]}  C(q-1,\l)\,\AO_\l(q-1)(x_{\perp y})\,e^{-2\pi R(x_{\perp y},\zeta)}\\
\nass
{}&=  2^{-\frac12}\,(x,\und{y})\, e^{-2\pi (x,\und{y})^2}\,\ph_{KM}^{\vhy,0}(x_{\perp y}).
\end{align*}
Passing to $\psi_{KM}(\tau,x)$ via  (\ref{psi0-to-psi}) and noting that 
$$Q(x) = -(x,\und{y})^2 + Q(x_{\perp y}),$$ we obtain the claimed formula. 
\end{proof}

%
%\cutter
%{\it Old version}
%
%\begin{prop} For $x\in V$, write $x = -(x,\uC_j)\,\uC_j + x_{\perp j}$, so that $x_{\perp j}$ is the $V_j$-component of $x$. Then \hfb
%(i)
%\begin{align*}
%\kappa_j^*\psi_{KM}^0(x) &= 2^{-\frac12}\,(-1)^{j}\,(x,\uC_j)\, e^{-2\pi (x,\uC_j)^2}\,\ph_{KM}^{V_j,0}(x_{\perp j}).\\
%\noalign{\noindent
%(ii)}
%\kappa_j^*(\psi_{KM}(\tau,x)) &= 2^{-\frac12}\,(-1)^{j}\, v^{\frac32}\,(x,\uC_j)\,e^{-2\pi v(x,\uC_j)^2} \qq^{-\frac12(x,\uC_j)^2}\, \ph_{KM}^{V_j}(\tau,x_{\perp j}).
%\end{align*}
%Here $ \ph_{KM}^{V_j,0}(\tau,\cdot)$ is the $\ph_{KM}^0$ Schwartz $(q-1)$-form on $D_j$.
%\end{prop} 
%\begin{proof} 
%The map on tangent spaces is given by 
%$$d\kappa_j:T_{\zeta}(\OFD_j) \lra T_{\kappa_j(\zeta)}(\OFD), \qquad \eta=[\eta_1,\dots, \eta_{q-1}] \mapsto [\eta_1,\dots, \eta_{j-1},0,\eta_j, \dots, \eta_{q-1}],$$
%and this map is compatible with the `horizontal' subspaces. 
%It follows that any term in $\psi_{KM}^0(x)$ involving an index $s=j$ in the differential form will vanish under pullback.  Thus we have
%\begin{align*}
%\kappa_j^*\psi_{KM}^0(x) &= 2^{\frac{q}2-1}\,(-1)^{j}\,(x,\uC_j)\, e^{-2\pi (x,\uC_j)^2}\sum_{\l=0}^{[(q-1)/2]}  C(q-1,\l)\,\AO_\l(q-1)(x_{\perp j})\,e^{-2\pi R(x_{\perp j},z_1)}\\
%\nass
%{}&=  2^{-\frac12}\,(-1)^{j}\,(x,\uC_j)\, e^{-2\pi (x,\uC_j)^2}\,\ph_{KM}^{V_j,0}(x_{\perp j}).
%\end{align*}
%Passing to $\psi_{KM}(\tau,x)$ via  (\ref{psi0-to-psi}) and noting that 
%$$Q(x) = -(x,\uC_j)^2 + Q(x_{\perp j}),$$ we obtain the claimed formula. 
%\end{proof}
%
%\cutter

Next consider the $(q-1)$-form $\Psi_{KM}^0(x)$.  Using the expressions just found and Lemma~\ref{lem2.3}, we have the following.

\begin{cor}\label{lem1.5}
On the subset of $D(\vhy)$ for which $\kappa_y(z) \notin D_x$, 
$$\kappa_y^*\Psi_{KM}^0(x) = -2^{\frac12}\,(x,\und{y})\int_1^\infty e^{-2\pi t^2 (x,\und{y})^2}\,\ph_{KM}^{\vhy,0}(t x_{\perp y})\, dt.$$
\end{cor}

In the next section, it will be useful to have the following variant, which involves a shift in the orientations. For an index $j$, $1\le j\le q$, define
\beq\label{defkappaj-jay}
\kappa_y[j]:  \OFD(\vhy)\hookrightarrow \OFD, \qquad \zeta \mapsto [\zeta_1, \dots, \zeta_{j-1},\und{y},\zeta_j, \dots, \zeta_{q-1}],
\eeq
and write $\kappa_y[j]:D(\vhy)\lra D$ for the corresponding embedding of symmetric spaces.  Of course, $\kappa_y = \kappa_y[1]$ and, 
the embeddings of symmetric spaces only depend on the parity of $j$. 

\begin{cor}\label{lem1.6} (i) 
On the subset of $D'_y$ for which $\kappa_j(z) \notin D_x$, 
$$\kappa_y[j]^*\Psi_{KM}^0(x) = (-1)^{j}\,2^{\frac12}\,(x,\und{y})\int_1^\infty e^{-2\pi t^2 (x,\und{y})^2}\,\ph_{KM}^{\vhy,0}(t x_{\perp y})\, dt.$$
(ii) On $D'_y$, 
$$\kappa_y^*(\psi_{KM}(\tau,x)) = (-1)^{j-1} 2^{-\frac12}\, v^{\frac32}\,(x,\und{y})\,e^{-2\pi v(x,\und{y})^2} \qq^{-\frac12(x,\und{y})^2}\, 
\ph_{KM}^{\vhy}(\tau,x_{\perp y}).$$
\end{cor}

\section{Proof of Theorem~\ref{main.theo.1}}\label{proof-Thm-4.1}

For convenience, we remove a factor independent of $z$ and write
$$\ph_{KM}(x) = \ph_{KM}^0(x) \,e^{-\pi (x,x)}.$$
In this section, we compute the cubical integrals
$$I^0(x;\CC) = \int_{S(\CC)} \ph_{KM}^0(x).$$

\subsection{The regular case}

First suppose that $x$ is regular with respect to $\CC$, so that, by Lemma~\ref{old-lemma},  the intersection $D_x\cap S(\CC)$ is either empty or 
consists of a single interior point $\rho_\CC(s(x))$ depending on whether $\P_q(x,\CC)$ vanishes or not. 
%
% intersection consists of a single point
%$$z_0 = z_0(x;\CC) = \sspan\{B_1(s_{01}), \dots, B_q(s_{0q})\}_{\text{p.o.}},$$
%where 
%$$s_{0j} = \frac{(x,C_j)}{(x,C_j) - (x,C_{j'})}.$$
%This point lies in the interior of $S(\CC)$. Indeed, to recall the argument, we require that, for all $j$, 
%$$0 = (x,B_j(s_j)) = (1-s_j)(x,C_j) + s_j (x, C_{j'}).$$
%If $(x,C_j)\ne (x,C_{j'})$, then $s_{0j}$ is the unique solution, whereas, if $(x,C_j)= (x,C_{j'})$, we would have $(x,C_j)=0$, contradicting regularity. 
%If $(x,C_j)$ and $(x,C_{j'})$ have opposite signs, then $s_{0j}\in (0,1)$, again due to regularity. 
If $\P_q(x,\CC)\ne0$ and for $\e>0$ sufficiently small, define a collection
$$\CC^\e(x) = \{\{B_1(s(x)_{1}-\e), B_1(s(x)_{1}+\e)\}, \dots,\{B_q(s(x)_{q}-\e), B_q(s(x)_{q}+\e)\}\}.$$
For simplicity, we will abbreviate this as
$$\CC^\e=\CC^\e(x)= \{\{C^\e_1,C^\e_{1'}\},\dots,\{C^\e_1,C^\e_{q'}\}\}.$$
The following result illustrates the convenience of the `good position' formulation. 
\begin{lem}
The collection $\CC^\e(x)$ is in good position. 
\end{lem}
\begin{proof} We note that, for $t\in [0,1]$, 
%\begin{align*}
$$(1-t) C^\e_j + t C^e_{j'} = (1-s(x)_j+\e-2t\e)C_j + (s(x)_j-\e+2t\e)C_{j'}$$
so that, for $t\in [0,1]^q$, 
$$\rho_{\CC^\e(x)}(t) = \rho_{\CC}(s(x)-\e+2\e t)\in D,$$
i.e., $\CC^\e(x)$ is in good position.
\end{proof}

By construction, % The collection $\CC^\e$ satisfies the same incidence conditions as $\CC$, and 
the singular $q$-cube $S(\CC^\e(x))$ contains the point $D_x\cap S(\CC)$. 
For $x$ regular with respect to $\CC$ and $\P_q(x,\CC)=0$, we let $S(\CC^\e(x))$ be the empty set. In general, we let
$$S^\e(x;\CC) = S(\CC) - \text{int}\ S(\CC^\e).$$
Then Stokes' Theorem and the inductive relation of Corollary~\ref{lem1.5},
imply the following inductive formula.  

\begin{prop}\label{inductive.prop} Suppose that $x$ is regular with respect to $\CC$.  Then the set $D_x$ does not meet $\d S(\CC)$,  the integral
$$I^{00}(x;\CC) := \int_{\d S(\CC)} \Psi_{KM}^0(x)$$
is well defined, and
$$I^0(x;\CC) = I^{00}(x;\CC) - \lim_{\e\downarrow 0}\,I^{00}(x;\CC^\e(x)).$$
Moreover,
\begin{align}\label{basic.integral}
I^{00}(x;\CC) &= 2^{\frac12}\,\sum_{j=1}^q (x,\uC_j)\bigg(\ \int_1^\infty e^{-2\pi t^2 (x,\uC_j)^2}\,I^0(t x_{\perp j};\CC[j])\, dt\bigg)\\
\nass
{}& \qquad\qquad - (x,\uC_{j'})\bigg(\ \int_1^\infty e^{-2\pi t^2 (x,\uC_{j'})^2}\,I^0(t x_{\perp j'};\CC[j'])\, dt\ \bigg).\notag
\end{align}
where $\CC[j]$ and $\CC[j']$ are given by (\ref{C[j]}) and (\ref{C[j']}).
\end{prop}
\begin{proof} 
Combining (\ref{bd-S}), (\ref{wall-front}), (\ref{wall-back}), and Corollary~\ref{lem1.6}, 
we obtain 
\begin{align*}
I^{00}(x;\CC) &= \sum_{j=1}^q (-1)^j\bigg(\int_{\d_j^+ S(\CC)} \Psi_{KM}^0(x) - \int_{\d_j^- S(\CC)} \Psi_{KM}^0(x)\ \bigg)\\
\nass
&= \sum_{j=1}^q (-1)^j\bigg(\int_{ S(\CC[j])} \kappa_j^*\Psi_{KM}^0(x) - \int_{S(\CC[j'])}\kappa_{j'}^* \Psi_{KM}^0(x)\ \bigg)\\
\nass
{}&=2^{\frac12}\sum_{j=1}^q (x,\uC_j)\bigg(\ \int_1^\infty e^{-2\pi t^2 (x,\uC_j)^2}\,I^0(t x_{\perp j};\CC[j])\, dt\bigg)\\
\nass
{}& \qquad\ \ \, - (x,\uC_{j'})\bigg(\ \int_1^\infty e^{-2\pi t^2 (x,\uC_{j'})^2}\,I^0(t x_{\perp j'};\CC[j'])\, dt\ \bigg),
%\nass
%{}& \qquad\qquad - (x,\uC_{j'})\bigg(\ \int_1^\infty e^{-2\pi t^2 (x,\uC_{j'})^2}\,I^0(t x_{\perp j'};\CC[j'])\, dt\ \bigg).
%\nass
%{}& \quad -\lim_{\e\downarrow 0} \bigg(\ 2^{\frac32} \sum_{j=1}^q (x,\uC^\e_j)\bigg(\ \int_1^\infty e^{-2\pi t^2 (x,\uC^\e_j)^2}\,I^0(t x_{\perp j};\CC^\e[j])\, dt\bigg)\\
%\nass
%{}& \qquad\qquad\qquad - (x,\uC^\e_{j'})\bigg(\ \int_1^\infty e^{-2\pi t^2 (x,\uC^\e_{j'})^2}\,I^0(t x_{\perp j'};\CC^\e[j'])\, dt\ \bigg) \ \bigg),
\end{align*}
as claimed.
%\begin{align*}
%I^0(x;\CC) &= \lim_{\e\downarrow 0} \int_{\d S^\e(x,\CC)}\Psi^0_{KM}(x)\\
%\nass
%{}&= 2^{\frac32}\sum_{j=1}^q (x,\uC_j)\bigg(\ \int_1^\infty e^{-2\pi t^2 (x,\uC_j)^2}\,I^0(t x_{\perp j};\CC[j])\, dt\bigg)\\
%\nass
%{}& \qquad\qquad - (x,\uC_{j'})\bigg(\ \int_1^\infty e^{-2\pi t^2 (x,\uC_{j'})^2}\,I^0(t x_{\perp j'};\CC[j'])\, dt\ \bigg)\\
%\nass
%{}& \quad -\lim_{\e\downarrow 0} \bigg(\ 2^{\frac32} \sum_{j=1}^q (x,\uC^\e_j)\bigg(\ \int_1^\infty e^{-2\pi t^2 (x,\uC^\e_j)^2}\,I^0(t x_{\perp j};\CC^\e[j])\, dt\bigg)\\
%\nass
%{}& \qquad\qquad\qquad - (x,\uC^\e_{j'})\bigg(\ \int_1^\infty e^{-2\pi t^2 (x,\uC^\e_{j'})^2}\,I^0(t x_{\perp j'};\CC^\e[j'])\, dt\ \bigg) \ \bigg),
%\end{align*}
%as claimed. Here we have changed variables in the $dt$ integral. Note that the factor $2^{\frac32}$ combines the $2^{\frac12}$ in Corollary~\ref{lem1.5} with the $2$ coming from the change of variables. 
\end{proof}

%{\bf The case $q=1$.}\hfb
\subsection{The case $q=1$}
As a basis for the inductive proof of Theorem~\ref{main.theo.1}, we first suppose that $q=1$, so that 
$\sig(V) = (m-1,1)$. This case is discussed in several places, \cite{kudla.zweg}, \cite{FK-I}, \cite{livinskyi}, etc., but we give the calculation for convenient reference. We have  
$$D\simeq \{ \zeta\in V\mid Q(\zeta) = -1\}, \qquad z = \sspan\{\zeta\}_{\text{p.o.}},$$
and the tangent space at $z\in D$ is
$$T_z(D) \simeq U(z) := z^\perp.$$
For any $x\in V$ the $1$-form $\o(1)$ on $D$ is defined by 
$$\o(1)_z(\eta) = (x,\eta),\qquad \eta\in U(z) \simeq T_z(D),$$
and the Schwartz form is given by
$$\ph_{KM}^0(x) = 2^{\frac12}\,\o(1)\,e^{-2\pi R(x,z)},$$
with $R(x,z) = (x,\zeta)^2$.
 Take $C$, $C'\in V$ such that 
$$Q(C)<0, \quad Q(C')<0, \quad (C,C')<0,$$
where the third condition insures that 
$$\{C\}_{\text{p.o.}} \simeq \uC = C \,|(C,C)|^{-\frac12},\qquad \{C'\}_{\text{p.o.}}\simeq \uC'$$
 lie on the same component of $D$. 
For $s\in [0,1]$, we define
$$B(s) = (1-s)C+ s C',$$
and note that 
$$(B(s),B(s)) = (1-s)^2(C,C) + 2 s (1-s) (C, C') + s^2 (C',C')<0,$$
so that the collection $\CC = \{\{C,C'\}\}$ is in good position. 
Writing
$$\zeta=\zeta(s) = B(s)|(B(s),B(s))|^{-\frac12},$$
we obtain a geodesic curve
$$\phi_{\CC}: [0,1] \lra D, \qquad s \mapsto \{B(s)\}_{\text{p.o.}}\simeq \zeta(s)$$
joining $\uC$ and $\uC'$. 
The  tangent vector to this curve will be $\dot \zeta= \frac{d}{ds}\zeta$, and 
\begin{align*}
I^0(x;\CC) &= 2^{\frac12}\,\int_0^1 (x,\dot\zeta(s))\,e^{-2\pi (x,\zeta(s))^2}\,ds\\
\nass
{}&= 2^{\frac12}\,\int_0^1 \frac{\d }{\d s}\left( -\int_{(x,\zeta(s))}^\infty e^{-2\pi t^2}\, dt\right)\,ds\\
\nass
{}&=2^{\frac12}\, \left(\int_{(x,\uC)}^\infty e^{-2\pi t^2}\, dt  - \int_{(x,\uC')}^\infty e^{-2\pi t^2}\, dt  \right).
\end{align*}
Since 
$$\int_u^\infty e^{-2\pi t^2}\,dt = 2^{-\frac32}( 1 - E(u\sqrt{2})),$$
for 
$$E(u) = 2\int_0^u e^{-\pi t^2}\,dt = 2\,\sgn(u)\int_0^{|u|} e^{-\pi t^2}\,dt,$$
as in \cite{zagier.bourb}, we obtain the expression 
\begin{align*}
I^0(x;\CC) &=\frac12\,\big(\,E((x,\uC')\sqrt{2}) - E((x,\uC)\sqrt{2})\,\big)\\
\nass
{}&=\frac12\,\big(\,E_1(C',x\sqrt{2}) - E_1(C,x\sqrt{2})\,\big),
\end{align*}
which is the $q=1$ case of Theorem~\ref{main.theo.1}.
Here we use the fact that, for $C\in V$ with $Q(C)<0$, a simple calculation shows that 
$E_1(C;x) = E((x,\uC))$.
Note that in this calculation we have not used the Stokes' theorem argument.  However, it is instructive to note that  
$$\psi_{KM}^0(x) = 2^{-\frac12}\,(x,\zeta)\,e^{-2\pi (x,\zeta)^2},$$
so that, for $z = \sspan\{C\}_{\text{p.o.}}\in D-D_x$, the primitive is given by
\begin{align*}
\Psi^0_{KM}(x) &=  -2^{-\frac12}\,(x,\uC)\int_1^\infty e^{-2\pi t (x,\uC)^2} \,t^{-\frac12}\,dt\\
\nass
{}&= -2^{\frac12}\,(x,\uC)\int_1^\infty e^{-2\pi t^2 (x,\uC)^2} \,dt\\
\nass
{}&= -\sgn(x,\uC)\int_{\sqrt{2}|(x,\uC)|}^\infty e^{-\pi t^2 } \,dt\\
\nass
{}&=\frac12\, \sgn(x,\uC)\bigg(2\int_0^{\sqrt{2}|(x,\uC)|} e^{-\pi t^2} \,dt -1\bigg)\\
\nass
{}&=\frac12\big(\ E_1(C;x \sqrt{2}) - \sgn(x,C) \,\big).
%
%\text{\rm erf}(\sqrt{2\pi}|(x,\zeta)|)\,)\\
%\nass
%{}&=\sgn(x,\zeta) - E_1((x,\zeta)\sqrt{2}),
\end{align*}
Thus the Stokes' theorem calculation gives
$$I^{00}(x;\CC)=\int_{\d S(\CC)} \Psi^0_{KM}(x) =  \frac12\big(\  E_1(C_{1'};x \sqrt{2}) -  \sgn(x,C_{1'}) - E_1(C_{1};x \sqrt{2}) + \sgn(x,C_{1})\,\big),$$
so that the basis for Zwegers `completion' construction emerges.

\subsection{Induction} 
Next we consider the inductive step.  Note that we are assuming that $x$ is regular with respect to $\CC$ so that (\ref{basic.integral}) holds, and we suppose that 
the identity (\ref{inductive.conjecture}) holds for all $q'<q$ and all $\CC'$ in good position. 
Let $I[j]$ and $I[j']$ be subsets of $\{1, \dots,\widehat{j},\dots,q\}$ and let $\bC[j]^{I[j]}$ (resp. $\bC[j']^{I[j']}$) be obtained by the recipe defining $\bC^I$ in 
Theorem~\ref{main.theo.1}, starting with the set $\CC[j]$ defined in (\ref{C[j]})  (resp. the set $\CC[j']$ defined in (\ref{C[j']}) \,). 
 Then (\ref{basic.integral}) becomes
\begin{align}\label{basic.integral.inductive}
I^{00}(x/\sqrt{2};\CC) &=\sum_{j=1}^q (x,\uC_j)\bigg(\ \int_1^\infty e^{-\pi t^2 (x,\uC_j)^2}\,I^0(t x_{\perp j}/\sqrt{2};\CC[j])\, dt\bigg)\notag\\
\nass
{}& \qquad\qquad - (x,\uC_{j'})\bigg(\ \int_1^\infty e^{-\pi t^2 (x,\uC_{j'})^2}\,I^0(t x_{\perp j'}/\sqrt{2};\CC[j'])\, dt\ \bigg)\notag \\
\nass
&=(-1)^{q-1}2^{1-q}\sum_{j=1}^q\bigg(\  \sum_{I[j]} (-1)^{|I[j]|}\,\bigg(\ (x,\uC_j) \int_1^\infty e^{-\pi t^2 (x,\uC_j)^2}\,E_{q-1}(\bC[j]^{I[j]}; t x_{\perp j})\, dt\ \bigg)\\
\nass
&{}\qquad -\sum_{I[j']} (-1)^{|I[j']|}\,\ \bigg((x,\uC_{j'}) \int_1^\infty e^{-\pi t^2 (x,\uC_{j'})^2}\,E_{q-1}(\bC[j']^{I[j']};  t x_{\perp j'})\, dt\ \bigg)\ \bigg).\notag
\end{align}

We want to compare this to the expression
$$-2^{-q}\sum_I (-1)^{|I|}\, E_q(\bC^I;x).$$
The key is to relate the individual quantities
$E_q(\bC^I;x)$ in this sum
and the terms  on the right side of (\ref{basic.integral.inductive}) where $I =  I[j]$ or $I = \{j\}\cup I[j']$.  
Note that, if $I=I[j]$ then the collection $\bC[j]^{I[j]}$ spans a negative $q-1$-plane in $V_j$ which maps to $z^I$ under $\kappa_j$.  
Similarly, if $I= \{j\}\cup I[j']$, then the collection $\bC[j']^{I[j']}$ spans a negative $q-1$-plane in $V_{j'}$ which maps to $z^I$ under $\kappa_{j'}$.
Thus, we are collecting all of the terms which `correspond to' a given vertex of the $q$-cube $S(\CC)$.  
The required identities are all consequences of that for $I=\emptyset$, and thus the main identity needed is the following. 
\begin{prop}\label{key.nazar} Suppose that $x$ is regular with respect to $\bC$. Then
%\footnote{\color{red} When I wrote out the proof of (5.8), I seem to have made a sign error, 
%explained by the fact that the function in \cite{nazar} is really $(-1)^eE_q(\CC,x)$. This will feed back into the induction and yield the 
%`extra' red sign in the main theorem. !! PLEASE CHECK!!} 
\beq\label{nazar-main}
E_q(\bC;x)-\sgn(\bC;x) = -2\sum_{j=1}^q  (x,\uC_j) \int_1^\infty e^{-\pi t^2 (x,\uC_j)^2}\,E_{q-1}(\bC[j]; t x_{\perp j})\, dt   ,
\eeq
where $\bC = \{C_1,\dots,C_q\}$,  $\bC[j] = \{C_{1\perp j},\dots,\widehat{C_j},\dots, C_{q\perp j}\}$, and $\sgn(\bC;x)$ is defined in (\ref{def.sgnCx}).
\end{prop}
\begin{rem} This result is just an integrated version of equation (25) in  Proposition~3.6 in \cite{nazar}.  For convenience, we give the proof, taken from \cite{nazar}, 
in our notation. 
\end{rem}
\begin{proof}
Let $Z$ be the negative $q$-plane spanned by $\bC=\{C_1,\dots,C_q\}$, and, for $y$, $y'\in Z$, let %$\lbrak y,y'\rbrak = -(y,y')$. 
$(\!(y,y')\!) = -(y,y')$. 
We also suppose that $x = \pr_Z(x)$. 
If $f$ is a smooth function on $Z$, then 
\begin{align}\label{euler-id}
-\int_1^\infty (\!(\nabla f(t\,x), x)\!)\,dt &= -\int_1^\infty \frac{d}{d t}\{f(tx)\}\,dt\\
\nass
{}&= f(x)- \lim_{t\rightarrow \infty} f(tx).\notag
\end{align}
Here $\nabla$ is the gradient operator and we assume that the radial limit of $f$ exists. 
On the other hand, by (25) Proposition~3.6 of \cite{nazar}, 
\beq\label{naza-1}
-(\!(\nabla E_q(\bC;x), x)\!) = 2 \sum_j  (\!(x,\uC_j)\!)\, e^{-\pi\,(\!(x,\uC_j)\!)^2}\, E_{q-1}(\bC[j]; x_{\perp j}).
\eeq
Moreover, for $x$ regular with respect to $\bC$, we have,  \cite{ABMP} and \cite{nazar}, Remark p.7, 
\beq\label{naza-2}
\lim_{t\rightarrow \infty} E_q(\bC;t x) = \sgn(\bC;x).
\eeq
For convenience, we will give the proof of (\ref{naza-1}) in Appendix I. Combining them and noting that the identity (\ref{euler-id}) is valid for the function $f(x)=E_q(\bC;x)$ when $x$ is regular with respect to $\bC$,  we have
$$E_q(\bC;x)-\sgn(\bC; x)  = -2 \sum_j  (x,\uC_j)\,  \int_1^\infty e^{-\pi\,t^2(x,\uC_j)^2}\, E_{q-1}(\bC[j]; tx_{\perp j})\,dt,$$
as required.  
%Here we want to double check the sign!
%\begin{align*}
%\int_1^\infty (\nabla E_q(\bC;x), x)\,dt &= E_q(\bC;x)- \sgn(\bC;x)
%\end{align*}
%
%\cutter
\end{proof}

\begin{cor}\label{main-cor}
$$I^{00}(x;\CC) =(-1)^q 2^{-q} \sum_{I}  (-1)^{|I|}\big(\ E_q(\bC^I;x\sqrt{2}) -\sgn(\bC^I; x)\ \big),$$
and
$$ I^0(x;\CC)  =  I^{00}(x;\CC)   + (-1)^q \P_q(x;\CC) = (-1)^q 2^{-q} \sum_{I}  (-1)^{|I|}E_q(\bC^I;x\sqrt{2}).$$
\end{cor}

Note that the second identity in Corollary~\ref{main-cor} follows from Proposition~\ref{inductive.prop}, since the first identity implies that 
$$\lim_{\e\downarrow 0}\,I^{00}(x;\CC^\e(x)) = -(-1)^q \P_q(x;\CC).$$
The identity of Theorem~\ref{main.theo.1} follows immediately from this %, the second expression in Proposition~\ref{inductive.prop},
and the continuity of $E(\bC^I;x)$ with respect to $\bC^I$. 

%Finally, we let
%$$I(\tau,x,\CC) = \int_{[0,1]^q} \phi_{\CC}^*(\ph_{KM}(\tau,x)).$$

\begin{cor}\label{cor7.6}
\begin{align}\label{basic-formula-1}
I(\tau,x,\CC) &:= \int_{[0,1]^q} \phi_{\CC}^*(\ph_{KM}(\tau,x))\\
\nass
{}&= \qq^{Q(x)}\,(-1)^q 2^{-q} \sum_{I}  (-1)^{|I|}E_q(\bC^I;x\sqrt{2v})\notag
\end{align}
\end{cor}

\section{Shadows of indefinite theta series}\label{section8}

In this section, we compute the shadow of $I_{\mu}(\tau,\CC)$, i.e., the complex conjugate of its image under the lowering operator
$-2i v^2 \frac{\d}{\d \bar \tau}$.  The crucial facts are the relation (i) of Lemma~\ref{lem2.3}, 
$$-2i v^2 \frac{\d}{\d \bar\tau}\ph_{KM}(\tau,x) = d \psi(\tau,x),$$% = v\,d\psi^0_{KM}(v^{\frac12}x)\,\qq^{Q(x)},$$
and the pullback identity (ii) of Lemma~\ref{lem1.6}. Then, as in the proof of Proposition~\ref{inductive.prop}, we have
\beq\label{shadow-1}
-2i v^2 \frac{\d}{\d \bar \tau}\{ I_\mu(\tau,\CC)\}  = \sum_{x\in \mu+L} \int_{\d S(\CC)} \psi(\tau,x),
\eeq
and 
\begin{align} \label{shadow-2}
\int_{\d S(\CC)} \psi(\tau,x)&= 2^{-\frac12}\,v^{\frac32} \sum_{j=1}^q \bigg(\ (x,\uC_{j'})\,e^{-2\pi v (x, \uC_{j'})^2}\,\qq^{-\frac12(x,\uC_{j'})^2}\,\,I(\tau,x_{\perp j'}, \CC[j'])\\
\nass
{}&\qquad \qquad \qquad
- (x,\uC_{j})\,e^{-2\pi v (x, \uC_{j})^2}\,\qq^{-\frac12(x,\uC_j)^2}\,I(\tau, x_{\perp j},\CC[j])\ \bigg).\notag
\end{align}
where we use the notation introduced in Corollary~\ref{cor7.6}.
The combination of (\ref{shadow-1}), (\ref{shadow-2}) and (\ref{basic-formula-1})
yields an explicit formula for the shadow of $I_\mu(\tau,\CC)$, a (typically non-holomorphic) modular form of 
weight $2 - \frac{p+q}2$.

Now suppose that the collection $\CC$ is rational. For each $j$, write
$$L_j^0 = L\cap \Q C_j, \quad L_j^1 = L\cap V_j$$
so that, for suitable coset representatives $\mu^0_{j,r}\in (L^0_j)^\vee$ and $\mu_{j,r}^1\in (L_j^1)^\vee$, 
$$\mu+L = \bigsqcup_r \bigg( (\mu_{j,r}^0+L^0_j) \oplus (\mu_{j,r}^1+L_j^1)\bigg).$$

Then, writing $I_\mu(\tau,\CC,L)$ to make explicit the dependence on the lattice $L$,
\begin{align*}
-2i v^2 \frac{\d}{\d \bar \tau}\{I_\mu(\tau,\CC,L)\}& %2i v^2 \overline{\frac{\d}{\d \bar \tau}\{ I_\mu(\tau,\CC)\} } 
= 2^{-\frac12}\,\sum_{j=1}^q
\sum_r v^{\frac32} \,\overline{\theta_{\mu^0_{j,r}}(\tau,L^0_j)}\,I_{\mu^1_{j,r}}(\tau,\CC[j],L^1_j)\\
%\sum_{x\in \mu+L} (x,\uC_{j'})\,e^{-2\pi v (x, \uC_{j'})^2}\,\qq^{-\frac12(x,\uC_{j'})^2}\,\,I(\tau,x_{\perp j'}, \CC[j'])\\
\nass
{}&\qquad \qquad \qquad\qquad
- \sum_{r'}v^{\frac32} \,\overline{\theta_{\mu^0_{j',r'}}(\tau,L^0_{j'})}\,I_{\mu^1_{j',r'}}(\tau,\CC[j'],L^1_{j'}),
\end{align*}
where 
$$\theta_{\mu^0_{j,r}}(\tau,L^0_j)  = \sum_{x^0\in \mu^0_{j,r}+L_j^0} (x^0,\uC_j)\,\qq^{\frac12 (x^0,\uC_j)^2}$$ 
is a unary theta series of weight $\frac32$. 
Thus,  the image of $I_\mu(\tau, \CC)$ under the $\xi$-operator is a linear combination of products of unary theta series of weight $\frac32$ 
and the conjugates of indefinite theta series for spaces of signature $(p,q-1)$, as asserted in Corollary~\ref{cor1.3}.

\begin{rem}  While we have only discussed the shadows of `cubical' integrals, the case of simplical integrals can the 
treated in the same way. 
\end{rem}

%%%%%%%%%%%%  tetra III text

\section{The case of a simplex}

In this section, we work out the theta integral over a simplex. The general inductive procedure is the 
same as in the cubical case, but some interesting differences arise. 

\subsection{Some geometry}

For $V$ of signature $(p,q)$, we consider a collection of vectors
$$\CC = [C_0,\dots,C_q]$$
$C_i\in V$ with $(C_i,C_i)<0$. We suppose that, for all $j$, 
$$P_j=\sspan\{C_0, \dots, \widehat{C_j}, \dots, C_q\}$$
is a negative $q$-plane. 
We assume that the collection $\CC$ is linearly independent and let $U = \sspan(\CC)$. 
Note that $\sig(U)=(1,q)$, and let $D(U)$ be the space of oriented negative $q$-planes in $U$.

Let 
$$\CC^\vee = [C_0^\vee, \dots, C_q^\vee] = \CC\,(\CC,\CC)^{-1}$$
be the dual basis of $U$ with respect to $(\ ,\ )$.  
Since $C_j^\vee$ then spans $P_j^\perp$, we have $(C_j^\vee,C_j^\vee)>0$.
Let 
$$\Delta_q = \{s=[s_0,\dots,s_q]\in \R\mid \text{ $s_j\ge 0$, for all $j$, $ \sum_j s_j=1$}\},$$
and, for $s\in \Delta_q$, let 
$$C^\vee(s) = \sum_{j=0}^q s_j C_j^\vee = \CC^\vee\, {}^ts.$$
Note that $s_j = (C^\vee(s), C_j)$. 
We say that $\CC$ is in {good position} if 
$$0<(C^\vee(s),C^\vee(s)) = s (\CC^\vee,\CC^\vee){}^ts = s (\CC,\CC)^{-1}{}^ts$$ 
for all $s\in \Delta_q$. For example, if all entries of $(\CC^\vee,\CC^\vee)=(\CC,\CC)^{-1}$ 
are non-negative, then $\CC$ is in good position.%\footnote{This is the sufficient condition imposed by Zwegers in his Dublin lectures. Is it necessary?}.

Given $\CC$ in good position, we define 
$$z(s) = C^\vee(s)^{\perp}\in D,$$
with orientation $\nu_{z(s)}\in \bbigwedge^q z(s)$ defined by 
\beq\label{orient-1}
C^\vee(s)\wedge \nu_{z(s)} = \nu_U,
\eeq
where we have fixed an orientation 
$$\nu_U = C_0\wedge C_1\wedge\dots\wedge C_q$$ 
in $\bbigwedge^{q+1}U$. 
For example, 
$$z_j = z(0,\dots, 1, \dots,0) = (C_j^\vee)^\perp=\sspan\{C_0,\dots, \widehat{C_j},\dots, C_q\}$$
with orientation given as follows.
%$$C_j^\vee\wedge C_0\wedge \dots\wedge \widehat{C_j}\wedge\dots \wedge C_q = \e_j\,C_0\wedge C_1\wedge\dots\wedge C_q.$$
Let $R_j$ be the $j$th column of the matrix $(C,C)^{-1}$, so that 
\beq\label{dual-rel}
C_j^\vee = \CC\,R_j = \sum_{i=0}^q R_{ij}C_i.
\eeq
Then 
$$C_j^\vee\wedge C_0\wedge \dots\wedge \widehat{C_j}\wedge\dots \wedge C_q =(-1)^j R_{jj}\,C_0\wedge C_1\wedge\dots\wedge C_q.$$
Since $R_{jj} = (C_j^\vee, C_j^\vee)>0$, 
\beq\label{face-orient}
z_j =\sspan\{C_0,\dots, \widehat{C_j},\dots, C_q\}[j]
\eeq
where the `twist' $[j]$ indicates that the given basis gives $(-1)^j \nu_{z(s)}$. 

For example, for $q=1$ we have
\beq\label{q1-orients}
z_0 = \sspan\{C_1\}_{\text{p.o.}}, \qquad z_1 = \sspan\{-C_0\}_{\text{p.o.}}.
\eeq
In particular, good position requires $(C_0,C_1)>0$ in this case!
For $q=2$, we have
\beq\label{q2-orients}
z_0 = \sspan\{C_1,C_2\}_{\text{p.o.}}, \quad z_1=\sspan\{-C_0,C_2\}_{\text{p.o.}}, \qquad z_2 = \sspan\{C_0,C_1\}_{\text{p.o.}}.
\eeq

By construction, all the $z_j$'s lie in the same component of $D$ and, by linear independence, the map
$$\phi_\CC:\Delta_q \lra D, \qquad s\mapsto z(s)$$
is an embedding. Let $S(\CC)= \phi_{\CC}(\Delta_q)$ be its image. 
The $j$th face of this tetrahedron is given by restricting to the subset of $s$ with $s_j=0$, so that it is given as 
$$\{ z\in S(\CC) \mid (C^\vee(s),C_j)=0\} = \{z\in S\mid C_j\in z\}.$$
Moreover, in the image $U_j$ of $U$ under the projection to $V_j =C_j^\perp$, we have
$$[C_0^\vee,\dots, \widehat{C_j^\vee},\dots, C_q^\vee]$$
is the dual basis to 
$$\CC_{\perp j}:=[C_{0\perp j}, \dots, C_{q\perp j}].$$
Thus, up to orientation, to be discussed in a moment, the restriction of $\phi_{\CC}$ to a face of $\Delta_q$ is 
again a simplex $\phi_{\CC_{\perp j}}$ in $D(V_j)$!  Note that, in particular, $\CC$ in good position implies that $\CC_{\perp j}$ is in good position for all $j$. 

Next consider $S(\CC)\cap D_x$.  This set depends only on $\pr_U(x)$ and is given by 
$$S(\CC)\cap D_x 
=\begin{cases} 
S(\CC)\cap D(U)_{\pr_U(x)}& \text{ if $Q(\pr_U(x))>0$,}\\
\nass
\emptyset&\text{if $\pr_U(x)\ne0$ and $Q(\pr_U(x))\le0$,}\\
\nass
S(\CC)&\text{if $\pr_U(x)=0$.}
\end{cases}
$$
Here, when $Q(\pr_U(x))>0$ so that $\pr_U(x)$ is a positive vector in $U$,  $D(U)_{\pr_U(x)}$ is a pair of oriented negative $q$-planes in $U$ given by the orthogonal complement 
to $\pr_U(x)$ with its two orientations.  
One of these has orientation determined by $\pr_{U}(x)$ by the analogue of the recipe (\ref{orient-1}). 
Then $S(\CC)\cap D_x =\phi_\CC(s(x))$ is the same $q$-plane with orientation shifted by
$$\sgn(\pr_U(x),C^\vee(s(x)) = \sgn(x,C^\vee(s(x))).$$
To determine $s(x)$, we solve
$$\pr_U(x) = \l \,C^\vee(s), \qquad s\in \Delta_q,\quad \l \in \R^\times$$
i.e.,
$$(x,C_j) = \l\,s_j, \qquad 0\le j\le q.$$
The existence of a solution implies that $\sgn(x,C_j)$, if non-zero,  is  independent of $j$ and that 
\beq\label{def-lambda}
\sum_{j=0}^q (x,C_j) = \l.
\eeq
Thus we have the following simple description. 
\begin{lem} Suppose that $Q(\pr_U(x))>0$. If $\sgn(x,C_j)$ is independent of $j$ when it is non-zero, then 
$$S(\CC)\cap D_x =\phi_\CC(s(x)),$$  
where
$$s(x)_j = (x,C_j)\l(x;\CC)^{-1}$$
with
$$\l(x,\CC)=\sum_j (x,C_j).$$
Otherwise
$S(\CC)\cap D_x =\phi_\CC(s(x))$ is empty. 
\end{lem}

When  $S(\CC)\cap D_x$ is non-empty, we determine the intersection number of the oriented $q$-simplex $S(\CC)$ with the oriented codimension $q$
cycle $D_x$.  The claim is that this is determined by the sign of the inner product of $\pr_U(x)$ with $C^\vee(s(x))$. 
\begin{prop}\label{int-num-prop}  Let 
%$$\delta_q(x;\CC):=2^{-q(q+1)/2} \prod_{0\le i<j\le q}\big(\,1+\sgn(x,C_i)\,\sgn(x,C_j)\,\big).$$
%\beq\label{def-phi-tri}
$\P_q^{\triangle}(x,\CC)$ be as in (\ref{def-phi-tri}). 
%= 2^{-q-1}\bigg(\prod_{j=0}^q(1-\sgn(x,C_j)) + (-1)^q\prod_{j=0}^q(1+\sgn(x,C_j)\ \bigg).\eeq
Then, if $x$ is regular with respect to $\CC$,
\beq\label{int-num-form}
I(S(\CC),D_x) = \P_q^{\triangle}(x,\CC).
\eeq
\end{prop}
Suppose that $\pr_U(x)\ne 0$. Then $\P_q^{\triangle}(x,\CC)$ is non-zero precisely when all of the non-zero $\sgn(x,C_i)$'s coincide.  
Suppose further that $s(x)$ lies on $r$ `walls', i.e., that $r$ of the inner products $(x,C_j)$ vanish. Then
$$\P_q^{\triangle}(x,\CC) = 2^{-r} \,(-1)^q \sgn(\l(x,\CC))^q.$$
When $\pr_U(x)=0$, then $\P_q^{\triangle}(x,\CC)=2^{-q}$ for $q$ even and vanishes for $q$ odd.  
Note that, if $x$ is not regular with respect to $\CC$, then the intersection number is not defined. 
%When $x$ is regular with respect to $\CC$, then  $\P_q^{\triangle}(x,\CC)$ is either $0$ or $(-1)^q(\sgn(x,C_i))^q$. 
%If $s(x)$ lies on $r>0$ walls, then $\P_q^{\triangle}(x,\CC)$ becomes $2^{-r}$, while the intersection number is not defined. 

\begin{proof}
Recall that, if $\zeta\in \OFD$ is a properly oriented $q$-frame projecting to $z\in D$, then 
$T_z(D)\simeq U(z)^q$, where $U(z)=z^\perp$ in $V$.  Also note that, under this isomorphism, the natural metric on 
$T_z(D)$ is given by $(\!(\eta,\eta')\!) = -\tr((\eta_i,\eta'_j))$ where $\eta = [\eta_1,\dots,\eta_q]$ and 
$\eta'=[\eta'_1,\dots,\eta'_q]$.
For our fixed collection $\CC$ with $U=\sspan\{\CC\}$, 
we have an embedding $D(U) \lra D$, where $D(U)$ is the space of oriented negative $q$-planes in $U$. Recall that 
$\sig(U)=(1,q)$. For $z\in D(U)$, write $W(z)$ for its orthogonal complement in $U$. 
Again supposing that $\zeta\in\OFD$ with projection $z$ is given, we have 
$$T_z(D(U)) \simeq W(z)^q.$$
Note that $\dim W(z)=1$, and suppose that $w=w(z)$ is a properly oriented basis vector. 
Then $T_z(D(U))$ is spanned by the vectors $\tau_1(w)=[w,0,\dots,0]$, $\tau_2(w) = [0,w,0,\dots,0]$, etc. 
Similarly, if $z\in D_x$, then the normal subspace to $T_z(D_x)$ is spanned by the vectors $\tau_i(x)$, $1\le i\le q$. 
For $z = \phi_\CC(s(x))$, we have $w= C^\vee(s(x))$, and the intersection 
number %\footnote{Here we should be careful about the order of the product and the orientation of $D$!} 
of these two cycles is then given by 
\begin{align*}
\sgn(\!(\, \tau_1(x)\wedge\dots\wedge \tau_q(x), \tau_1(w)\wedge\dots\wedge\tau_q(w)\,)\!) &= (-1)^q\det((\tau_i(x),\tau_j(w)))\\
\nass
{}&=(-1)^q \sgn(x,C^\vee(s(x)))^q.
\end{align*}
But now 
$$C^\vee(s(x)) = \lambda(x,\CC)^{-1}\sum_j (x,C_j)\, C_j^\vee,$$
and, recalling (\ref{dual-rel}), 
$$(x,C^\vee(s(x))) = \lambda(x,\CC)^{-1}\sum_j (x,C_j) (x,C_j^\vee) = \lambda(x,\CC)^{-1}\sum_{i,j}(x,C_j) R_{i,j}(x,C_i).$$
If we assume that all of the non-zero $(x,C_i)$'s have the same sign, and recalling that $R_{i,j}\ge 0$, we see that 
$$\sgn(x,C^\vee(s(x))) = \sgn(\l(x,\CC)).$$
\end{proof}

%{\bf Question:}  Is there some geometric reason for the fact that this is automatically positive for $q$ even?

For $q=1$, and $x$ regular with respect to $\CC$, 
$$I(S(\CC),D_x)=-\frac12(\sgn(x,C_0)+\sgn(x,C_1)).$$
Note that, due to the `twist' occurring in (\ref{face-orient}), our negative lines are $z_0 = \sspan\{C_1\}_{\text{p.o.}}$ and $z_1 = \sspan\{-C_0\}_{\text{p.o.}}$ 
Thus the `cubical' data is $\CC^{\square} = \{C_1,-C_0\}$, and $I(S(\CC),D_x)$ coincides with 
$$\Psq_1(x,\CC^{\square}) = \frac12(\sgn(x,-C_0) - \sgn(x,C_1)).$$
% so that this intersection number {\bf does not coincide} with that in the cubical case. 

%The relevant cone is then the `Weyl chamber'
%$$\mathcal W(\CC) = \{ y\in U\mid (y,C_j)\ge 0, \forall j, y\ne 0\}.$$

\subsection{The integral of the theta form}

We would like to compute
$$I^0(x;\CC) = \int_{S(\CC)}\ph^0_{KM}(x).$$

The case $q=1$, coincides with the Zwegers case for $\CC^{\square} = \{C_1,-C_0\}$, and we have 
\beq\label{tetra:q=1}
I^0(x/\sqrt{2};\CC) =- \frac12(\, E_1(C_0;x)+E_1(C_1;x)).
\eeq
As a check on signs, note that, since 
$$\lim_{t\rightarrow \infty} E_1(C;tx) = \sgn(x,C),$$
this is consistent with the value of $I(S(\CC),D_x)$ for $q=1$ above.

For the general case, we suppose that $x$ is regular with respect to $\CC$ and proceed by induction. 
Due to regularity, $S(\CC)\cap D_x$ is either empty or is a single point $\phi_\CC(s(x))$ on the interior of $S(\CC)$. 
Recall that, by (\ref{face-orient}), 
$$\d S(\CC) = \sum_{j=0}^q (-1)^j S(\CC_{\perp j}).$$
 Then by Remark~3.4 of \cite{FK-I}, we have
\beq\label{pre-inductive}
I^0(x;\CC)=\int_{S(\CC)} \ph_{KM}^0(x) = I(S(\CC),D_x) + \int_{\d S(\CC)} \Psi^0_{KM}(x).
\eeq
Since   $\lim_{t\rightarrow \infty} \Psi^0_{KM}(tx)=0$, this identity gives the limiting value
$$\lim_{t\rightarrow \infty} I^0(tx;\CC)=\lim_{t\rightarrow \infty}\int_{S(\CC)} \ph_{KM}^0(tx) = I(S(\CC),D_x).$$ 

Now using Corollary~\ref{lem1.5}, we have the inductive formula
\begin{align}\label{MIF}
\int_{\d S(\CC)} \Psi^0_{KM}(x) &= \sum_{j=0}^q  (-1)^j \int_{S(\CC_{\perp j})} \kappa_j^*\Psi^0_{KM}(x)\\
\nass
{}&= \sum_{j=0}^q2^{\frac12}\,(x,\uC_j)\int_1^\infty e^{-2\pi t^2 (x,\uC_j)^2}\, I^0(t x_{\perp j};\CC_{\perp j})\, dt.\notag
\end{align}
 
Using this, we obtain the following explicit formula. 
\begin{theo} \label{tetra-theo} For a subset $I\subset \{0, 1, \dots, q\}$, let
$\CC^{(I)}$ be the collection of $q+1-|I|$ elements where the $C_i$ with $i\in I$ have been omitted. 
$$I^0(x/\sqrt{2};\CC) = (-1)^q 2^{-q}\sum_{r=0}^{[q/2]}\sum_{\substack{I \\ |I| = 2r+1}} E_{q-2r}(\CC^{(I)}; x).$$
Here $E_0(\dots)=1$. 
\end{theo}

\begin{rem}
(i) Note that if this formula is proved for $x$ regular, then it holds for all $x$ by continuity.\hfb 
(ii) Substituting $tx$ for $x$ and letting $t$ go to infinity, we obtain the `holomorphic' part:
\beq\label{hol-part}
(-1)^q 2^{-q}\sum_{r=0}^{[q/2]}\sum_{\substack{I \\ |I| = 2r+1}} \sgn(\CC^{I}, x),
\eeq
where $\sgn(\emptyset ,x)=1$.  In the case of $x$ regular,  (\ref{pre-inductive}) implies that this must coincide with $I(S(\CC),D_x)$.
In fact, it is easily checked that (\ref{hol-part}) is equal to $\P_q^{\triangle}(x,\CC)$ for all $x$. 
Thus our theta integral is the non-holomorphic completion of the series 
$$\sum_{x\in \mu+L} \P_q^{\triangle}(x,\CC) \, \qq^{Q(x)}.$$
%(iii) Note that, we can also write 
%$$\P_q^{\triangle}(x,\CC) = 2^{-q-1}\bigg(\prod_{j\in A}(1-\s_j) + (-1)^q\prod_{j\in A}(1+\s_j)\ \bigg).$$
%Let $B\subset A$ be the set of indices $j$ for which $\s_j\ne 0$ and suppose that 
%all of the non-zero $\s_j$'s have the same sign. 
%Then 
%$$\P_q^{\triangle}(x,\CC) = 2^{|B|-q-1}\,\lambda(x,\CC)^q,$$
\end{rem}

\begin{proof}  The case $q=1$ is (\ref{tetra:q=1}). In the induction, we use the notation
$$\CC[j] = [C_{0\perp j}, \dots, C_{j-1 \perp j}, C_{j+1\perp j}, \dots, C_{q\perp j}].$$
Let $A=\{0,1,\dots,q\}$ and for a subset $I\subset A$, let 
$\CC^I$ be the collection of $q+1-|I|$ vectors obtained by omitting the $C_i$ with $i\in I$. 
Also denote by $I[j]$ a subset of $A[j]:=\{0,1, \dots, \hat{j}, \dots,q\}$. 

We have
\begin{align*}
I^0(x/\sqrt{2};\CC) - I(S(\CC),D_x){} &= \int_{\d S(\CC)} \Psi^0_{KM}(x/\sqrt{2})\\
\nass
{}&=\sum_{j=0}^q(x,\uC_j)\int_1^\infty e^{-\pi t^2 (x,\uC_j)^2}\, I^0( t x_{\perp j}/\sqrt{2};\CC[j])\, dt
\end{align*}
\begin{align*}
{}&=(-1)^q 2^{-q} \sum_{j=0}^q-2(x,\uC_j)\int_1^\infty e^{-\pi t^2 (x,\uC_j)^2}\, 
\sum_{r=0}^{[(q-1)/2]}\sum_{\substack{I[j] \subset A[j]\\ |I[j]| = 2r+1}} E_{q-1-2r}(\CC[j]^{I[j]}, t x_{\perp j})\,dt\\
\nass
{}&=(-1)^q 2^{-q} \sum_{r=0}^{[(q-1)/2]}\sum_{j\in A} \sum_{\substack{I \subset A\\ |I| = 2r+1\\ j\notin I}} -2(x,\uC_j)\int_1^\infty e^{-\pi t^2 (x,\uC_j)^2}\, 
 E_{q-1-2r}(\CC[j]^{I}; t x_{\perp j})\,dt
 \end{align*}
 \begin{align*}
{}&=(-1)^q 2^{-q} \sum_{r=0}^{[(q-1)/2]}\sum_{\substack{I \subset A\\ |I| = 2r+1}}  \sum_{\substack{j\in A\\ j\notin I}} -2(x,\uC_j)\int_1^\infty e^{-\pi t^2 (x,\uC_j)^2}\, 
 E_{q-1-2r}((\CC^{I})[j]; t x_{\perp j})\,dt\\
\nass
{}&=(-1)^q 2^{-q} \sum_{r=0}^{[(q-1)/2]}\sum_{\substack{I \subset A\\ |I| = 2r+1}}\big(\    E_{q-2r}(\CC^{I}; x)  - \sgn(x,\CC^I)\ \big)\\
\nass
{}&=(-1)^q 2^{-q} \sum_{r=0}^{[q/2]}\sum_{\substack{I \subset A\\ |I| = 2r+1}}E_{q-2r}(\CC^{I}; x)\\
\nass
{}&\qquad\qquad\qquad- (-1)^q 2^{-q}\sum_{r=0}^{[(q-1)/2]}\sum_{\substack{I \subset A\\ |I| = 2r+1}} \sgn(x,\CC^I)  - (-1)^q 2^{-q}\,\delta_{q,\text{even}}.
\end{align*}
Thus, to finish the proof, we note that 
\beq\label{int-num-conj}
I(S(\CC),D_x)= (-1)^q 2^{-q}\sum_{r=0}^{[q/2]}\sum_{\substack{I \subset A\\ |I| = 2r+1}} \sgn(x,\CC^I),
\eeq
where we use the convention that $\sgn(x,\emptyset)=1$. 
Here recall that we are assuming that $x$ is regular with respect to $\CC$. 
To check this, observe that 
\begin{align*}
(-1)^q2^{-q}\sum_{r=0}^{[q/2]}\sum_{\substack{I \subset A\\ |I| = 2r+1}} \sgn(x,\CC^I)&= (-1)^q2^{-q}\sum_{\substack{J\subset A\\ |J|\equiv q (2)}} \prod_{j\in J}\s_j\\
\nass
{}&= (-1)^q 2^{-q-1}\bigg(\prod_{j\in A}(1+\s_j) + (-1)^q\prod_{j\in A}(1-\s_j)\ \bigg).
\end{align*}
%and this coincides with $I(S(\CC),D_x)$ by Proposition~\ref{int-num-prop}.
%If $q$ is even, this quantity is $1$ if all of the $\s_j$'s have the same sign and vanishes otherwise. 
%If $q$ is odd, this quantity is $1$ if all $\s_j=1$, $-1$ if all $\s_j=-1$ and vanishes otherwise. 
%Comparing with Proposition~\ref{int-num-prop}, we obtain (\ref{int-num-conj}). 
\end{proof}

%%%%%%%%%%%%%%%  end tetraIII text

\section{An example}
In this section, we write out a very simple example, which illustrates the relation between the (degenerate) cubical formula and the 
simplicial formula in the case $q=2$.

Let $\mathcal A = \{A_0,A_1,A_2\}$ be the data for a $2$-simplex. The vertices are:
$$z_0 = \sspan\{A_1,A_2\}_{\text{p.o.}}, \quad z_1=\sspan\{-A_0,A_2\}_{\text{p.o.}}, \qquad z_2 = \sspan\{A_0,A_1\}_{\text{p.o.}},$$ 
and the theta integral is
$$\frac14\big(\ E_2(A_1,A_2)+E_2(A_0,A_2)+E_2(A_0,A_1)+1\ \big).$$

We can consider the related cubical data $\CC = \{\{C_1,C_{1'}\},\{C_2,C_{2'}\}\}$, where
$$C_1=A_0, \ 
C_2=A_1,\ 
C_{2'}=-A_2, \ 
C_{1'}=C_{2'}-C_2 = -A_1-A_2,
$$
%\begin{align*}
%C_1&=A_0\\
%C_2&=A_1\\
%C_{2'}&=-A_2\\
%C_{1'}&=C_{2'}-C_2 = -A_1-A_2.
%\end{align*}
so that the associated (degenerate) $2$-cube has vertices
$$z_2=\{C_1,C_2\}, \ z_1=\{C_1,C_{2'}\},\ z_0=\{C_{1'},C_{2'}\}=\{C_{1'},C_2\},$$
and theta integral
%$$\frac14\big(\ E_2(C_1,C_2)-E_2(C_1,C_{2'})-E_2(C_{1'},C_2)+E_2(C_{1'},C_{2'})\ \big).$$
%The vertices of the square are then
%
%and the theta integral becomes
\begin{align*}&\frac14\big(\ E_2(C_1,C_2)-E_2(C_1,C_{2'})-E_2(C_{1'},C_2)+E_2(C_{1'},C_{2'})\ \big)\\
%\nass
%{}=&\frac14\big(\ E_2(A_0,A_1)-E_2(A_0,-A_2)-E_2(-A_1-A_2,A_1)+E_2(-A_1-A_2,-A_2)\ \big)\\
\nass
{}=&\frac14\big(\ E_2(A_0,A_1)+E_2(A_0,A_2)+E_2(A_1+A_2,A_1)+E_2(A_1+A_2,A_2)\ \big).\\
\end{align*}
Coincidence of the two theta integrals is the equivalent to the identity
$$E_2(A_1+A_2,A_1)+E_2(A_1+A_2,A_2) = E_2(A_1,A_2)+1,$$
where all terms are given by integrals over the negative $2$-plane $z_0$. 
Writing $y\in z_0$ as $y= a A_1^\vee+bA_2^\vee$, with respect to the dual basis, 
and noting that
$$\sgn(a+b)(\sgn(a)+\sgn(b)) = \sgn(a)\sgn(b) +1,$$
for $a$ and $b$ not both $0$, the identity follows.

\section{Appendix I: Some proofs and details}\label{append-2}

\subsection{Proof of part (iii) of Lemma~\ref{old-lemma}}
%In this section, we prove part (iii) of Lemma~\ref{old-lemma}.  So s
Suppose that $\CC$ is in good position and that $x\in V$ with 
$\P_q(x;\CC)\ne 0$. Let $s_0 = s(x)$ be the unique point of $[0,1]^q$ such that $\rho_\CC(s_0) = D_x\cap S(\CC)$. 
Note that the map $\rho_\CC$ extends to an open neighborhood of $[0,1]^q$ so that, even if $s_0$ lies on the boundary, we can define $\rho_\CC$ 
on an open set $\mathcal U$ around $s_0$.  We lift $\rho_\CC$ to a map $\tilde\rho_\CC:\mathcal U \rightarrow \OFD$, 
defined by 
$$\tilde \rho_\CC: s\mapsto \zeta(s) = B(s) P^{-1}, \qquad P\in \Sym_q(\R)_{>0}, \quad P^2 = -(B(s),B(s)).$$
For convenience, we write $B=[B_1,\dots, B_q]=B(s)$. 
%We note that the basis of $\rho_\CC(s)$ dual to $B$ with respect to $-(\,,\,)$ is
%$$B^\vee=[B^\vee_1,\dots, B^\vee_q] = -B \, (B,B)^{-1} ,$$
%and that $(B^\vee, B^\vee) = (B,B)^{-1}$. 
Then
$$(\tilde\rho_\CC)_*(\frac{\d}{\d s_j}) = \dot B_j P^{-1} - \zeta \dot P_j P^{-1},\qquad \dot B_j := \frac{\d}{\d s_j} B= [0,\dots, -C_j+C_{j'},\dots, 0], \quad \dot P_j := \frac{\d}{\d s_j} P.$$
The components in the connection subspace $U(z)^q$ of $T_\zeta(\OFD)$ are then
$$(\rho_\CC)_*(\frac{\d}{\d s_j})= \tau_j\,P^{-1}, \qquad  \tau_j =[0,\dots, \pr_{U(z)}(-C_j+C_{j'}),\dots, 0] $$
and these are linearly independent provided $\pr_{U(z)}(-C_j+C_{j'})\ne0$ for all $j$. 
But at the point $z_0=\rho_\CC(s_0)$, we have $x\in U(z_0)$, and the $q$ vectors 
$$\eta(x,j)=[0,\dots,0,,x,0,\dots,0],$$
with $x$ in the $j$th component, span the normal to $T_{z_0}(D_x)$. Note that the metric $g$ on $T_z(D)\simeq U(z)^q$ 
is given by 
$$g(\eta,\eta') = \tr( \,(\eta_i,\eta'_j)\,).$$
Then we have  
$$
g\big(\eta(x,i),\tau_j\big) = [\ (x,C_{j'}) - (x,C_j)\ ]\,\delta_{ij}.
$$

This shows that $\tau_j\ne0$ for all $j$ and hence $\rho_\CC$ is immersive at $s(x)$. We can choose the open neighborhood $\mathcal U$ of $s(x)$ in $\R^q$ 
so that the restriction of $\rho_\CC$ to $\mathcal U$ is an embedding. 
The orientation of the codimension $q$ cycle $D_x$ is defined by an element of $\nu_{z,x}\in \bbigwedge^{(p-1)q}T_z(D_x)$
such that 
$$\nu_x\wedge \nu_{z,x}\in \sideset{}{^{pq}}\bbigwedge(T_z(D))$$
is properly oriented, where 
$$\nu_x = \eta(x,1)\wedge \dots \wedge\eta(x,q).$$
Here we have fixed an orientation of $D$. 
Thus the intersection number at $z_0$ of $D_x$ with $\rho_\CC(\mathcal U)$ is 
$$I(D_x,\rho_\CC(\mathcal U)) = \sgn \det( g(\eta(x,i),\tau_j)) = \prod_j \sgn\big(\ (x,C_{j'}) - (x,C_j)\ \big).$$
If $x$ is regular with respect to $\CC$, then this quantity is 
$$2^{-q}\prod_{j=1}^q \big(\ \sgn(x, C_{j'}) - \sgn(x,C_j)\ \big) = (-1)^q \P_q(x;\CC).$$
In general, we have 
\beq\label{def-inter}
(-1)^q\P_q(x;\CC)  = 2^{-r}\,I(D_x,\rho_\CC(\mathcal U)),
\eeq
where $r$, $0\le r\le q$, is the number of walls passing through $s(x)$. Thus, $\P_q(x;\CC)$ is a `weighted' intersection number.

\subsection{ Proof of (\ref{naza-1})}

%\cutter

\newcommand{\lgs}{(\!(}
\newcommand{\rgs}{)\!)}

%\begin{proof}[Proof of (\ref{naza-1})]
%We temporarily change the sign of the inner product on 
For $y$, $y'\in Z= \sspan\{C\}$, we write $\lgs y,y'\rgs = -(y,y')$,  and we assume that $x\in Z$.
We let 
$$C^\vee = [C_1^\vee, \dots, C_q^\vee] = C \lgs C,C\rgs^{-1}$$
be the dual basis. We write 
$$x= \sum_i x_i\,C_i^\vee, \qquad x_i = \lgs x,C_i\rgs .$$
For a fixed index $j$, we write
$$x = x_{\perp j} + x'\,C_j,\qquad  x_{\perp j} = \sum_{i\ne j} x_i C_i^\vee,\qquad x_j=\lgs x, C_j\rgs  = x' \lgs C_j, C_j\rgs,$$
and similarly for our variable of integration $y\in Z$. Note that, in particular, 
$$\sgn\lgs y,C_j\rgs = \sgn(y').$$
We can write 
$$dy = dy_{\perp j} \, dy'$$ 
where
$$1=\int_{Z} e^{-\pi \lgs y,y\rgs}\,dy = \int_{Z_{\perp j}}\int_{\R} e^{-\pi \lgs y_{\perp j},y_{\perp j}\rgs} \,e^{-\pi (y')^2\lgs C_j,C_j\rgs}\,dy_{\perp j}\,dy' ,$$
where 
$dy'$ is $\lgs C_j,C_j\rgs^{\frac12} $ times Lebesque measure, so that 
$$\int_\R e^{-\pi (y')^2\lgs C_j,C_j\rgs}\,dy'  = 1.$$
We write\footnote{Note that the extra factor of $(-1)^q$ etc. is due to our temporary change in the sign of the inner product on $Z$, 
so that our $E_q$ differs from that in \cite{nazar} by this sign.}
\begin{align*}
(-1)^q E_q(C;x) &= \int_Z e^{-\pi\lgs y-x,y-x\rgs}\,\prod_i\sgn\lgs y,C_i\rgs\,dy\\
\nass
{}&= \int_{Z_{\perp j}}\int_{\R} e^{-\pi \lgs y_{\perp j}-x_{\perp j},y_{\perp j}-x_{\perp j}\rgs } \,e^{-\pi (y'-x')^2\lgs C_j,C_j\rgs}\,\prod_{i\ne j}\sgn\lgs y,C_i\rgs\,\sgn(y')\,dy_{\perp j}\,dy' \\
\nass
{}&= (-1)^{q-1}E_{q-1}(C[j];x_{\perp j}) \,\int_{\R} e^{-\pi (y'-x')^2\lgs C_j,C_j\rgs}\,\sgn(y')\,dy'\\
\nass
{}&= (-1)^{q-1}E_{q-1}(C[j];x_{\perp j}) \,\int_{\R} e^{-\pi (y')^2\lgs C_j,C_j\rgs}\,\sgn(y_j+x_j)\,dy' .
\end{align*}
But then, taking into account that $dy' = \lgs C_j,C_j\rgs^{-\frac12}\,d_{\text{Leb}}y_j$, we have
\begin{align*}
x_j\,\frac{\d }{\d x_j}\big\{\,(-1)^q E_q(C;x)\,\big\} &= \lgs x,C_j\rgs\,(-1)^{q-1}E_{q-1}(C[j];x_{\perp j})\\
\nass
&\qquad\qquad\times \,\int_{\R} e^{-\pi y_j^2\lgs C_j,C_j\rgs^{-1}}\,2\delta(y_j+x_j)\,\lgs C_j,C_j\rgs^{-\frac12}\,d_{\text{Leb}}y_j\\
\nass
{}&=2\lgs x,\uC_j\rgs (-1)^{q-1} \,E_{q-1}(C[j];x_{\perp j})\,e^{-\pi \lgs x,\uC_j\rgs^2}.
\end{align*}
Here recall that $\uC_j = C_j \lgs C_j,C_j\rgs^{-\frac12}$.  Summing on $j$, we obtain (\ref{naza-1}).

\end{document}